\documentclass[a4paper, 12pt]{article} 

\usepackage[utf8]{inputenc}
\usepackage[T1]{fontenc}
\usepackage[a4paper, margin=2.5cm]{geometry}
\usepackage{needspace, mathscinet}

\usepackage{libertine} 
\usepackage{inconsolata} 

\usepackage{authblk, amsmath, amsthm, amssymb, enumerate, graphicx, mathtools}
\usepackage[textsize=small, textwidth=2cm, color=yellow]{todonotes}
\usepackage{thmtools, thm-restate}
\usepackage[colorlinks=true, citecolor=blue, linkcolor=blue, urlcolor=blue]{hyperref}

\declaretheorem[name=Theorem, numberwithin=section]{theorem}
\declaretheorem[name=Lemma, sibling=theorem]{lemma}

\declaretheorem[name=Corollary, sibling=theorem]{corollary}
\declaretheorem[name=Conjecture, sibling=theorem]{conjecture}
\declaretheorem[name=Thomassen's Conjecture, sibling=theorem, numbered=no]{Thomassen}
\declaretheorem[name=Circle Packing Theorem, sibling=theorem, numbered=no]{circlePacking}
\declaretheorem[name=Problem, sibling=theorem]{problem}

\declaretheorem[name=Observation, style=remark, sibling=theorem]{observation}

\def\cqedsymbol{\ifmmode$\lrcorner$\else{\unskip\nobreak\hfil
\penalty50\hskip1em\null\nobreak\hfil$\lrcorner$
\parfillskip=0pt\finalhyphendemerits=0\endgraf}\fi}

\let\le\leqslant
\let\ge\geqslant
\let\leq\leqslant
\let\geq\geqslant

\let\olditemize\itemize \renewcommand{\itemize}{\olditemize\itemsep0pt}

\let\OLDthebibliography\thebibliography
\renewcommand\thebibliography[1]{
  \OLDthebibliography{#1}
  \setlength{\parskip}{0pt}
  \setlength{\itemsep}{0pt plus 0.3ex}
}%
\AtBeginDocument{
   \def\MR#1{}
}%

\newcommand{\F}{\mathcal{F}}

\title{A survey of degree-boundedness}
\author[1]{Xiying Du}
\author[2]{Rose McCarty\thanks{The second author is supported by the National Science Foundation under Grant No. DMS-2202961.}}
\affil[1]{School of Mathematics, Georgia Institute of Technology. E-mail: \href{mailto:xdu90@gatech.edu}{\texttt{xdu90@gatech.edu}}}
\affil[2]{School of Mathematics and School of Computer Science, Georgia Institute of Technology. E-mail: \href{mailto:rmccarty3@gatech.edu}{\texttt{rmccarty3@gatech.edu}}}
\begin{document}
\maketitle

\begin{abstract}
Suppose a graph has no large balanced bicliques, but has large minimum degree. Then what can we say about its induced subgraphs? This question motivates the study of degree-boundedness, which is like $\chi$-boundedness but for minimum degree instead of chromatic number. We survey this area with an eye towards open problems.
\end{abstract}

\tableofcontents

\section{Introduction}

Informally, a class of graphs $\mathcal{F}$ is \emph{degree-bounded} if every graph in $\mathcal{F}$ with sufficiently large minimum degree has a big balanced biclique $K_{t,t}$ as a subgraph. This intuitively says that for graphs in $\mathcal{F}$, minimum degree is a ``local property''. This is definitely not the case for all graphs due to a famous random construction of Erd\H{o}s~\cite{ErdosLargeChi} and lower bounds on the K\H{o}v\'{a}ri-S\'{o}s-Tur\'{a}n Theorem~\cite{KST} on the extremal numbers of bipartite graphs. In fact, degree-boundedness has many nice connections to extremal graph theory, and we discuss these in Section~\ref{sec:denseGraphs}. 


For now though, we focus on what has recently become the standard definition of degree-boundedness. Given a graph $G$, the \emph{biclique number} of $G$ is the largest integer $t$ so that $G$ has a subgraph (not necessarily induced) which is isomorphic to the biclique $K_{t,t}$. We denote the biclique number of $G$ by $\tau(G)$. Then a class of graphs $\mathcal{F}$ is \emph{degree-bounded} if there exists a function $f$ so that each graph $G \in \mathcal{F}$ has a vertex of degree at most $f(\tau(G))$. (Just to avoid technicalities, we assume throughout the paper that every graph has at least one vertex.) We write $\delta(G)$ for the \emph{minimum degree} of a graph $G$; thus we have $\delta(G) \leq f(\tau(G))$ for every $G \in \mathcal{F}$. Any function $f$ with this property is called a \emph{degree-bounding function} for $\mathcal{F}$.

Culminating a long line of research~\cite{polyBoundPath, polyBoundednessBip, singlyExponential, kiersteadPenrice, KLST20, inducedBip, polyBoundTree}, Gir{\~{a}}o and Hunter~\cite{polyBoundedness} recently proved the following surprising theorem. It says that under a mild assumption about the class $ \mathcal{F}$ being \emph{hereditary}, that is, closed under taking induced subgraphs, there must exist a very efficient degree-bounding function for $ \mathcal{F}$.

\begin{restatable}[Gir{\~{a}}o and Hunter~\cite{polyBoundedness}]{theorem}{mainTheorem}
\label{thm:main}
Any hereditary, degree-bounded class $\mathcal{F}$ has a degree-bounding function that is a polynomial.
\end{restatable}
\noindent The theorem is striking because it says that as soon as any bound at all exists, there is actually a very efficient bound. We discuss how one proves something like Theorem~\ref{thm:main} in Section~\ref{sec:ratesOfGrowth}.

The definition of degree-boundedness is intentionally chosen to resemble the definition of $\chi$-boundedness. The study of $\chi$-boundedness has a rich history, from Vizing's Theorem~\cite{vizing} about line graphs to the Strong Perfect Graph Theorem of Chudnovsky, Robertson, Seymour, and Thomas~\cite{strongPerfect}. In general, $\chi$-boundedness studies when the chromatic number -- rather than the minimum degree -- is a ``local property''. (The \emph{chromatic number} of a graph $G$ is denoted by $\chi(G)$, and it is the minimum number of colors needed to color the vertices of $G$ so that no pair of adjacent vertices receives the same color.) Since cliques rather than bicliques are the natural ``local certificates'' of large chromatic number, $\chi$-boundedness instead considers the {clique number} of a graph $G$.  The \emph{clique number} of $G$ is denoted by $\omega(G)$, and it is the largest integer $t$ so that the clique $K_t$ is a subgraph of $G$. A
class of graphs $\mathcal{F}$ is \emph{$\chi$-bounded} if there exists a function $g$ so that every graph $G \in \mathcal{F}$ has chromatic number at most $g(\omega(G))$. Any such function is called a \emph{$\chi$-bounding function} for~$\mathcal{F}$.

So degree-boundedness is like $\chi$-boundedness, but where we study the minimum degree instead of the chromatic number, and bicliques instead of cliques. Given the similarity in the definitions, another thing that makes Theorem~\ref{thm:main} particularly surprising is that the analogous statement for $\chi$-boundedness is false. In fact it is very false; Bria\'{n}ski, Davies, and Walczak~\cite{notPolyChi} recently proved that there exists a hereditary, $\chi$-bounded class whose optimal $\chi$-bounding function grows arbitrarily quickly. This theorem was a huge step forward in our understanding of $\chi$-boundedness. Previously, Esperet~\cite{esperet2017habilitation} had conjectured that there must be a $\chi$-bounding function that is a polynomial, and there was a lot of evidence for this conjecture. In Section~\ref{sec:ChiBoundedness}, we discuss how Theorem~\ref{thm:main} is related to Esperet's Conjecture, to $\chi$-boundedness, and to another famous conjecture of Erd\H{o}s and Hajnal~\cite{EHConjecture}. In fact, there is a reasonably general setting under which Theorem~\ref{thm:main} actually recovers Esperet's Conjecture; see Section~\ref{subsec:cliqueNum}.

Finally, in Section~\ref{sec:example} we discuss examples of degree-bounded classes. Many -- but not all -- of these classes are also $\chi$-bounded. The examples typically come from forbidding some substructure, or considering graphs with geometric representations, or bounding some measure of ``width'', or from connections to finite model theory. The last two cases are conjecturally related via flip-width, which is a new notion of width introduced by Toru\'{n}czyk~\cite{flipWidth}. Flip-width generalizes the notion of the cops and robbers game which was used by Seymour and Thomas~\cite{copsAndRobbersTW} to characterize graphs of bounded treewidth. We refer the reader to~\cite{flipWidth}. 

Some illustrative examples of degree-bounded classes are given in the theorem below.

\begin{theorem}[by \cite{KO04induced}, Lemma~\ref{lem:intBalls}, \cite{flipWidth}, and~\cite{KLST20, inducedBip}, respectively]
\label{thm:examples}
If $\mathcal{F}$ is a class of graphs such that any of the following properties hold, then $\mathcal{F}$ is degree-bounded.
\begin{enumerate}
    \item There exists a graph $H$ so that no graph in $\mathcal{F}$ contains an induced subdivision of $H$.
    \item There exists a positive integer $d$ so that every graph in $\mathcal{F}$ is an intersection graph of a collection of balls in $\mathbb{R}^d$.
    \item The class $\mathcal{F}$ has bounded flip-width. 
    \item There exists an integer $c$ so that no graph in $\mathcal{F}$ has an induced subgraph which is $4$-cycle free and has minimum degree at least~$c$.
\end{enumerate}
\end{theorem}
\noindent Given graphs $G$ and $H$, we say that $G$ is \emph{$H$-free} if $G$ has no subgraph (not necessarily induced) which is isomorphic to $H$. Note that the sufficient condition from part \emph{4.} of Theorem~\ref{thm:examples} is also necessary; so it actually characterizes degree-bounded classes. This is one of the key facts which underlies Theorem~\ref{thm:main}, as we will discuss in Section~\ref{sec:ratesOfGrowth}.

We hope that this survey paper serves to introduce new researchers to the exciting area of degree-boundedness. With this goal in mind, we focus on the state of the art and on future directions of research. We draw attention to a number of conjectures from other papers (see Conjectures~\ref{conj:favorite}, \ref{conj:extremal}, \ref{conj:EHEquiv}, and~\ref{conj:oddSubdivision}), pose open problems (see Problems~\ref{prob:neighbrReg}, \ref{prob:chiAndDeg}, \ref{prob:rControlled}, and~\ref{prob:spherePacking}), and propose new conjectures (Conjectures~\ref{conj:VM} and~\ref{conj:intMinorFree}).

\section{What makes a graph dense?}
\label{sec:denseGraphs}

Degree-boundedness is motivated by the following question: What are the unavoidable induced subgraphs in graphs with large minimum degree? In this section we take a broader approach to this question. In Section~\ref{sec:ThomassensConj}, we discuss Thomassen's Conjecture about unavoidable subgraphs, how to generalize this conjecture to induced subgraphs, and its connections to degree-boundedness. Then in Section~\ref{sec:strongerExtremal} we discuss extremal aspects related to the K\H{o}v\'{a}ri-S\'{o}s-Tur\'{a}n Theorem~\cite{KST} about the number of edges in a graph with no $K_{t,t}$-subgraph. We reinterpret degree-bounded classes as precisely those classes which satisfy much stronger extremal bounds than normal. Let us also point the reader to the recent survey by Smorodinsky~\cite{zarankiewiczSurvey} about such extremal problems.

Throughout this section, we freely use the fact that we can go back and forth between minimum degree and average degree. (The \emph{average degree} of a graph $G$ is just the average degree of its vertices, so it is $\sum_{v \in V(G)}\deg(v)/|V(G)|$.) That is, we will use the fact that having an induced subgraph of large minimum degree is the same thing as having an induced subgraph of large average degree. In particular, the following lemma implies that a hereditary class of graphs $\mathcal{F}$ is degree-bounded if and only if there exists a function $f$ so that every graph $G \in \mathcal{F}$ has average degree at most $f(\tau(G))$.

\begin{lemma} 
\label{lem:minDegreeAvgDeg}
For every $d \in \mathbb{R}$, every graph with minimum degree at least $d$ has average degree at least $d$, and every graph with average degree at least $d$ has an induced subgraph with minimum degree at least $d/2$.
\end{lemma}
\begin{proof}
The first part of the lemma is obvious, and for the second part of the lemma, suppose that $G$ is a graph of average degree at least $d$. Let $H$ be an induced subgraph of $G$ which has as few vertices as possible subject to still having average degree at least $d$. If $v$ was a vertex of $H$ of degree at most $d/2$, then the graph $H-v$ would still have average degree at least $d$, a contradiction to the choice of $H$.
\end{proof}

\subsection{Thomassen's Conjecture}
\label{sec:ThomassensConj}

A classical result of Erd\H{o}s~\cite{ErdosLargeChi} implies that there exist graphs of arbitrarily large minimum degree and girth. (The \emph{girth} of a graph $G$, denoted by $\textrm{girth}(G)$, is the minimum number of edges in any cycle of $G$, or $\infty$ if $G$ has no cycles. So a graph has large girth if all of its cycles are long.) Erd\H{o}s' construction is random, so in some sense these graphs are quite common. 

In 1983, Thomassen~\cite{Thomassen83} proposed the following beautiful conjecture, which says that not only are these graphs common, but they are in fact entirely unavoidable: they occur inside every graph of sufficiently large minimum degree.

\begin{Thomassen}[\cite{Thomassen83}]
There is a function $f:\mathbb{N}^2 \rightarrow \mathbb{N}$ so that for any integers $k$ and $d$, every graph $G$ with $\delta(G) \geq f(k,d)$ has a subgraph $H$ so that $\delta(H) > d$ and $\textrm{girth}(H) > k$.
\end{Thomassen}

This conjecture has remained wide open for the last $40$ years. In particular, it is open when $k=6$. Thomassen's Conjecture can, however, be reduced to the case that $G$ is bipartite by applying the following well-known fact.

\begin{lemma}
\label{lem:bipSubgraph}
Every graph $G$ has a bipartite subgraph $H$ which contains at least half of the edges of~$G$.
\end{lemma}
\noindent This lemma can be proven by taking a max cut in $G$. We can prove that the max-cut contains at least $|E(G)|/2$ edges by applying induction on the number of vertices of $G$, or by including each vertex in one of the two sides independently at random with probability~$1/2$.

Lemma~\ref{lem:bipSubgraph} yields the $k=3$ case of Thomassen's Conjecture, and it shows that we only care about even $k$. Already the next case when $k=4$ is much more difficult, and it was proven in 2004 by K\"{u}hn and Osthus~\cite{KO04}.

\begin{theorem}[\cite{KO04}]
\label{thm:KO}
There is a function $f_{\ref{thm:KO}}:\mathbb{N} \rightarrow \mathbb{N}$ so that for any integer $d$, every bipartite graph $G$ with $\delta(G)\geq f_{\ref{thm:KO}}(d)$ has a subgraph $H$ with $\delta(H) \geq d$ and without $4$-cycles.
\end{theorem}

\noindent Dellamonica, Koubek, Martin, and R\"{o}dl~\cite{DKMR11} gave another proof of Theorem~\ref{thm:KO} using a theorem of F\"{u}redi~\cite{Furedi83} on $r$-uniform hypergraphs. This theorem of F\"{u}redi can be translated to the setting of bipartite graphs where one side is $r$-regular, that is, contains only vertices of degree $r$. Essentially, F\"{u}redi's Theorem gives us some ``uniform behavior'' for the number of common neighbors of pairs of vertices on the other side. Montgomery, Pokrovskiy, and Sudakov~\cite{MPS2021} also gave another proof of Theorem~\ref{thm:KO}, obtaining better quantitative bounds.

There is not much more known about Thomassen's Conjecture. It is true for regular graphs, and for graphs whose maximum degree is not much more than their average degree, using the original proof of Erd\H{o}s~\cite{ErdosLargeChi}. That is, we can find the desired subgraph $H$ of a regular graph $G$ by including each edge of $G$ independently at random with some probability $p$, and then deleting one edge from each short cycle. Greatly improving on this result, Dellamonica and R\"{o}dl~\cite{DR11} proved Thomassen's Conjecture for graphs $G$ whose maximum degree is at most doubly exponential in a small enough power of their average degree. There has also been some work on stronger versions of Thomassen's Conjecture for expanders, where we can hope to find the desired graph $H$ on the same vertex-set as $G$; we refer the reader to~\cite{BFK2022}.

Unfortunately, there is an obstruction to all of these approaches that use regularity; it is not possible to reduce to the case that $G$ is even remotely regular, due to the following theorem of Pyber, R\"{o}dl, and Szemer\'{e}di~\cite{PRS95} and its generalizations.

\begin{theorem}[\cite{PRS95}]
\label{thm:notRegular}
There exist bipartite graphs with $n$ vertices and at least $\Omega(n\log\log(n))$ edges which do not have any $3$-regular subgraph.
\end{theorem}
Note that, since these graphs are bipartite, they also do not have any $r$-regular subgraph for any $r\geq 3$. The idea of the proof of Theorem~\ref{thm:notRegular} is to take the union of $\log\log(n)$ random star forests whose high-degree vertices are all on the same side. In each of the star forests, all of the high-degree vertices have the same degree, and this degree is very different for different star forests. Dellamonica, Koubek, Martin, and R\"{o}dl~\cite[Theorem~17]{DKMR11} used a similar construction to show that, for any fixed integer $\Delta$, there are $n$-vertex graphs with $\Omega_\Delta(n\log\log(n))$ edges that have no subgraph of minimum degree at least $4$ and maximum degree at most $\Delta$.

It is natural to ask whether Thomassen's Conjecture holds for \emph{induced} subgraphs. Here, the answer is clearly no since the clique $K_t$ and the biclique $K_{t,t}$ have no induced subgraph of minimum degree more than $1$ and girth more than $4$. However, the second author conjectures that these are the only obstructions~\cite[Conjecture~5]{inducedBip}.

\begin{conjecture}[\cite{inducedBip}]
\label{conj:favorite}
There is a function $g:\mathbb{N}^2 \rightarrow \mathbb{N}$ so that for any integers $k$ and $d$, every graph $G$ with $\delta(G) \geq g(k,d)$ has an induced subgraph $H$ so that either\begin{itemize}
    \item $H$ is isomorphic to $K_d$,
    \item $H$ is isomorphic to $K_{d,d}$, or
    \item $\delta(H) > d$ and $\textrm{girth}(H) > k$.
\end{itemize}
\end{conjecture}
\noindent Conjecture~\ref{conj:favorite} implies Thomassen's Conjecture because large cliques and bicliques have subgraphs of large minimum degree and girth (because they are regular, for instance).

Surprisingly, there is almost as much evidence for Conjecture~\ref{conj:favorite} as for Thomassen's Conjecture. First of all, Conjecture~\ref{conj:favorite} can be reduced to the case that $G$ is bipartite using the following beautiful theorem of Kwan, Letzter, Sudakov, and Tran~\cite{KLST20}. This theorem is the ``induced version'' of Lemma~\ref{lem:bipSubgraph}.

\begin{theorem}[\cite{KLST20}]
\label{thm:KLST}
There is a function $f_{\ref{thm:KLST}}:\mathbb{N} \rightarrow \mathbb{N}$ so that for any integer $d$, every graph $G$ with $\delta(G) \geq f_{\ref{thm:KLST}}(d)$ has an induced subgraph $H$ so that either \begin{itemize}
    \item $H$ is isomorphic to $K_d$, or
    \item $\delta(H) > d$ and $H$ is bipartite.
\end{itemize}
\end{theorem}
\noindent This theorem proved a conjecture of Esperet, Kang, and Thomass\'{e}~\cite{EKT19} and answered a question of K\"{u}hn and Osthus~\cite{KO04induced}. The bounds obtained by Kwan, Letzter, Sudakov, and Tran~\cite{KLST20} are also quite good, but we defer quantitative aspects until Section~\ref{sec:ratesOfGrowth}.

Next, generalizing Theorem~\ref{thm:KO} of K\"{u}hn and Osthus~\cite{KO04}, the second author~\cite{inducedBip} proved the following result, which takes care of the $k=4$ case of Conjecture~\ref{conj:favorite}.

\begin{theorem}[\cite{inducedBip}]
\label{thm:indBip}
There is a function $f_{\ref{thm:indBip}}:\mathbb{N} \rightarrow \mathbb{N}$ so that for any integer $d$, every bipartite graph $G$ with $\delta(G)\geq f_{\ref{thm:indBip}}(d)$ has an induced subgraph $H$ so that either\begin{itemize}
    \item $H$ is isomorphic to $K_{d,d}$, or
    \item $\delta(H) > d$ and $H$ has no $4$-cycles.
\end{itemize}
\end{theorem}
\noindent The proof of Theorem~\ref{thm:indBip} follows the ideas of Dellamonica, Koubek, Martin, and R\"{o}dl~\cite{DKMR11} and also uses F\"{u}redi's Theorem~\cite{Furedi83} about uniform hypergraphs.

Finally, as was the case for Thomassen's Conjecture, Conjecture~\ref{conj:favorite} also holds for regular graphs; see~\cite[Proposition~3]{inducedBip}. The proof works the same way as in the theorem of Erd\H{o}s~\cite{ErdosLargeChi}, except we work with vertices instead of edges. That is, let us assume that $G$ is a regular graph with large degree and with no induced subgraph isomorphic to $K_d$ or $K_{d,d}$. Then by Ramsey's Theorem, there is also an integer $t$ so that $G$ has no subgraph (not necessarily induced) which is isomorphic to $K_{t,t}$. We form an induced subgraph $H$ of $G$ by including each vertex independently at random with some probability $p$, and then deleting one vertex from each short cycle. Since $G$ has no subgraph isomorphic to $K_{t,t}$, we can apply the K\H{o}v\'{a}ri-S\'{o}s-Tur\'{a}n Theorem~\cite{KST}, which is stated as Theorem~\ref{thm:KST} in the next subsection. This theorem bounds the number of edges of a graph with no $K_{t,t}$-subgraph, which in turn lets us bound the number of short cycles of $G$. This is exactly what we need for the random construction of $H$ to work.

We note that every known proof of Theorem~\ref{thm:KO} of K\"{u}hn and Osthus~\cite{KO04} has some step where a $K_{t,t}$ subgraph is found. This seems to be the key obstacle in proving Thomassen's Conjecture. Conjecture~\ref{conj:favorite} suggests that the $k=6$ case is fundamentally different because there is no ``special'' induced subgraph like a clique or a biclique to be found.

So, what is the connection between Thomassen's Conjecture and degree-boundedness? It turns out that Theorems~\ref{thm:KLST} and~\ref{thm:indBip} are equivalent to the following statement about degree-boundedness.

\begin{corollary}
\label{cor:key}
A hereditary class of graphs $\mathcal{F}$ is degree-bounded if and only if there exists an integer $c$ so that every bipartite graph $H \in \mathcal{F}$ with no $4$-cycles has $\delta(H) \leq c$.
\end{corollary}

It just takes some unravelling of the definitions to see why Theorems~\ref{thm:KLST} and~\ref{thm:indBip} imply Corollary~\ref{cor:key}. To see why Corollary~\ref{cor:key} implies Theorem~\ref{thm:KLST}, suppose for a contradiction that there exist graphs of arbitrarily large minimum degree with no clique $K_d$ and no induced bipartite subgraph of minimum degree greater than $d$. Let $\mathcal{F}$ be the class consisting of all of these graphs and all of their induced subgraphs. Then there exists an integer $s$ so that no graph in $\mathcal{F}$ has $K_{s,s}$ as an induced subgraph. Thus, by Ramsey's Theorem, there exists an integer $t$ so that no graph in $\mathcal{F}$ has $K_{t,t}$ as a subgraph (not necessarily induced). So this class $\mathcal{F}$ is not degree-bounded, and by Corollary~\ref{cor:key} it contains bipartite graphs with arbitrarily large minimum degree, a contradiction. The proof that Corollary~\ref{cor:key} implies Theorem~\ref{thm:indBip} is similar.

Corollary~\ref{cor:key} provides the intuition for why every degree-bounded class has an ``efficient'' degree-bounding function. Essentially, the proofs all show that this constant $c$ ``controls'' the rate of growth of an optimal degree-bounding function. We discuss this more in Section~\ref{sec:ratesOfGrowth}. 

Let us conclude this section by pointing out an interesting open case of Conjecture~\ref{conj:favorite}. We say that a graph $G$ is \emph{strongly neighborhood-regular} if $G$ is a bipartite graph with bipartition $(A,B)$ so that every vertex in $A$ has the same degree, say $r$, and there exists a partition of $B$ into parts $B_1, B_2, \ldots, B_r$ so that\begin{itemize}
\item every vertex in $A$ has exactly one neighbor in each of $B_1, B_2, \ldots, B_r$, and
\item for each $i \in \{1,2,\ldots, r\}$, every vertex in $B_i$ has the same degree, say $d_i$.
\end{itemize}
\noindent This seems like the key case to study; it might even be possible to use regularity methods to reduce to something similar. It would be particularly interesting to know what happens when, informally, $r \ll d_1 \ll d_2 \ll \ldots \ll d_r$, that is, when the degrees are very different. We note that these graphs are essentially the ones used by Pyber, R\"{o}dl, and Szemer\'{e}di~\cite{PRS95} to prove Theorem~\ref{thm:notRegular}, and by Dellamonica, Koubek, Martin, and R\"{o}dl~\cite[Theorem~17]{DKMR11} to prove that Thomassen's Conjecture cannot be reduced to the case of nearly regular graphs.

\begin{problem}
\label{prob:neighbrReg}
Let $G$ be a strongly neighborhood-regular graph with sufficiently high degree and no $4$-cycles. Does $G$ have an induced subgraph $H$ with $\delta(H)>d$ and no $6$-cycles?
\end{problem}

\subsection{Stronger extremal bounds}
\label{sec:strongerExtremal}

A simple corollary of the K\H{o}v\'{a}ri-S\'{o}s-Tur\'{a}n Theorem~\cite{KST} is the following extremal bound on the number of edges in a graph with no subgraph isomorphic to $K_{t,t}$. (This is essentially equivalent to the Zarankiewicz problem, however, we state the theorem for general graphs instead of bipartite graphs. The general case can be reduced to the bipartite case by making a ``positive'' and a ``negative'' copy of each vertex, and replacing each edge $\{u,v\}$ by two edges between the positive and negative copies of $u$ and $v$.)

\begin{theorem}[\cite{KST}]
\label{thm:KST}
For any integer $t$, every $n$-vertex graph with no subgraph isomorphic to $K_{t,t}$ has at most $n^{2-1/t}+tn$ edges.
\end{theorem}

A key theorem that has come out of the work on degree-boundedness is an improved version of Theorem~\ref{thm:KST} for classes which forbid a fixed bipartite graph $H$ as an induced subgraph. Notice that for any hereditary, degree-bounded class $\mathcal{F}$, there exists an integer $c$ so that no graph in $\mathcal{F}$ contains an induced subgraph with no $4$-cycle-subgraphs and with minimum degree at least $c$. So we can take $H$ to be any bipartite graph that has no $4$-cycle-subgraphs and that has minimum degree at least $c$. For instance, we can take $H$ to be the incidence graph of a projective plane of sufficiently large order.

This next theorem was discovered independently by Bourneuf, Buci\'{c}, Cook, and Davies~\cite[Theorem~1.4]{polyBoundednessBip} and by Gir{\~{a}}o and Hunter~\cite[Lemma~7.1]{polyBoundedness} in connection to degree-boundedness. A very similar theorem was also proven by Fox, Pach, Sheffer, Suk, and Zahl~\cite{FPSSZ2017} when the graph under consideration is bipartite. This is related to the VC-dimension of a graph, and some nice quantitative improvements have been made since then; see~\cite{zarankiewiczsemialgebraic, zarankiewiczVCdim}. There is also a short proof using $\epsilon -t$ nets~\cite{KSTepsilonNets}, as introduced by Alon, Jartoux, Keller, Smorodinsky, and Yuditsky~\cite{epsilonTNets}. Tis type of problem was also considered by Loh, Tait, Timmons, and Zhou~\cite{LTTZ2018}, who proved something similar in the case that $H$ is a biclique.

\begin{restatable}[\cite{polyBoundednessBip, polyBoundedness}]{theorem}{strongerKST}
\label{thm:strongerKST} For each bipartite graph $H$, there exists a polynomial $p_H$ and a number $\epsilon_H>0$ so that for any integer $t$, every $n$-vertex graph with no induced subgraph isomorphic to $H$ and no subgraph isomorphic to $K_{t,t}$ has at most $p_H(t)n^{2-\epsilon_H}$ edges.
\end{restatable}

\noindent The key improvement in Theorem~\ref{thm:strongerKST} is that the amount saved in the exponent is fixed for fixed $H$. So when $t$ is large in comparison to $H$, the bound from Theorem~\ref{thm:strongerKST} is far better than the bound from Theorem~\ref{thm:KST}. 

For obtaining efficient degree-bounds, it is also very important that $p_H$ is a polynomial in $t$. From the extremal perspective, however, it makes sense to try and optimize $\epsilon_H$ while allowing $p_H$ to be any function in $t$. In this direction, Hunter, Milojevi\'{c}, Sudakov, and Tomon~\cite{KSTHereditaryConj} proposed the following intriguing conjecture. Given a graph $H$ and an integer $n$, we write $\textrm{ex}
(n,H)$ for the maximum number of edges in an $n$-vertex graph with no subgraph (not necessarily induced) that is isomorphic to $H$.

\begin{conjecture}[\cite{KSTHereditaryConj}]
\label{conj:extremal}
For each bipartite graph $H$, there exists a function $f_H$ so that every $n$-vertex graph with no induced subgraph isomorphic to $H$ and no subgraph isomorphic to $K_{t,t}$ has at most $f_H(t)\cdot \textrm{ex}
(n,H)$ edges.
\end{conjecture}

When $H$ is a fixed tree $T$, the extremal number $\textrm{ex}(n,T)$ is linear, and this conjecture is true precisely because the class of graphs with no induced subgraph isomorphic to $T$ is degree-bounded (as proven in~\cite{kiersteadPenrice, polyBoundTree, KSTHereditaryConj}, with better and better bounds). In fact, this approach of looking at extremal numbers gives us another way to think about degree-boundedness. Given a class of graphs $\mathcal{F}$ and a graph $H$, we write $\textrm{ex}_{\mathcal{F}}(n, H)$ for the maximum number of edges in an $n$-vertex graph in $\mathcal{F}$ that has no subgraph isomorphic to $H$.

\begin{observation}
\label{obs:equivDef}
A hereditary class of graphs $\mathcal{F}$ is degree-bounded if and only if $\textrm{ex}_{\mathcal{F}}(n, K_{t,t}) = \mathcal{O}_t(n)$ for each fixed integer $t$. 
\end{observation}

So, informally, a class is degree-bounded if its extremal numbers are linear for bipartite graphs. This way of looking at things suggests a more systematic study of classes which have smaller extremal numbers than normal. There are many results known in this direction, especially for classes which arise from discrete geometry and model theory; in addition to the examples discussed above, see, for instance, \cite{semilinear2021, KSTepsilonNets, incidenceHigherDim, Walsh2020}. Perhaps most famously, Szemer\'{e}di and Trotter~\cite{ST1983} proved that any incidence graph of $n$ points and $n$ lines in the real plane has $\mathcal{O}(n^{4/3})$ edges. This bound substantially beats the bound of $\mathcal{O}(n^{3/2})$ which holds for general $n$-vertex graphs with no subgraph isomorphic to the $4$-cycle $C_4$ (see Theorem~\ref{thm:KST}). For general graphs that bound is also tight, since incidence graphs of projective planes have $\Omega(n^{3/2})$ edges.

The work on degree-boundedness suggests the following question. If a hereditary class of graphs $\mathcal{F}$ satisfies stronger extremal bounds than normal for $\textrm{ex}_{\mathcal{F}}(n, K_{t,t})$, then does this ``propagate'' to better bounds for $\textrm{ex}_{\mathcal{F}}(n, K_{s,s})$ when $s \geq t$? One interesting regime is when the extremal numbers are {almost} linear. We say that a class of graphs $\mathcal{F}$ is \emph{almost degree-bounded} if $\textrm{ex}_{\mathcal{F}}(n, K_{t,t}) = n^{1+o_t(1)}$ for each fixed integer $t$. (That is, for each fixed integer $t$, the $o_t(1)$ term goes to zero as $n$ goes to infinity.) Our interest in this case stems primarily from connections to finite model theory. It is known that monadically dependent classes of graphs are almost degree-bounded; see~\cite{AdlerAdler} and~\cite[Corollary2.3]{NORS19}, as well as the discussion in Section~\ref{subsec:beyond}. Also, motivated by connections to model theory, Basit, Chernikov, Starchenko, Tao, and Tran~\cite{semilinear2021} proved that any class of semilinear bipartite graphs is almost degree-bounded. Improving upon a special case, Tomon and Zakharov~\cite{KSTBoxes2021} gave a simpler proof that for any fixed integer $d$, the class of all intersection graphs of axis-parallel boxes in $\mathbb{R}^d$ is almost degree-bounded.

Indeed, the following result is a corollary of a theorem of Du, Gir{\~{a}}o, Hunter, McCarty, and Scott~\cite{singlyExponential}. They proved that there exists a universal constant $C$ so that for all integers $k$ and $t$ and every graph $G$ with average degree at least $k^{Ct^3}$, either $G$ has a subgraph isomorphic to $K_{t,t}$, or $G$ has an induced subgraph $H$ with no $4$-cycles and with average degree at least $k$. Thus, by choosing $k$ appropriately, we obtain the following corollary.

\begin{corollary}[~\cite{singlyExponential}]
\label{cor:almostDegBounded}
A hereditary class of graphs $\mathcal{F}$ is almost degree-bounded if and only if $\textrm{ex}_{\mathcal{F}}(n, K_{2,2}) = n^{1+o(1)}$.
\end{corollary}

We note that Corollary~\ref{cor:almostDegBounded} actually has a much simpler proof. The idea is to apply a theorem of Erd\H{o}s and Simonovits~\cite{ErdosSimonovits} to pass to an induced subgraph which is almost regular. Then one can remove all $4$-cycles by choosing each vertex independently at random with an appropriately chosen probability $p$, and deleting one vertex from each remaining $4$-cycle. We do not explain this argument in detail; however, see Section~\ref{subsec:almostRegular} for more information about how to deal with nearly regular graphs.


\section{Connections to {$\chi$}-boundedness}
\label{sec:ChiBoundedness}

In this section we compare and contrast degree-boundedness with $\chi$-boundedness. We highlight many key parts of the theory of $\chi$-boundedness as we go, but for a more thorough introduction of $\chi$-boundedness, we refer the reader to the survey of Scott and Seymour~\cite{ss20survey}.

We note that there is one direct connection between minimum degree and chromatic number, which arises from the observation that the greedy algorithm can be used to color any $d$-degenerate graph with $d+1$ colors. (A graph $G$ is \emph{$d$-degenerate} if every induced subgraph of $G$ has minimum degree at most $d$.) We use this observation a few times in this section.

\subsection{Esperet's Conjecture}

Polynomial $\chi$-boundedness was a well-known open problem due to Esperet~\cite{esperet2017habilitation} which has since been resolved in the negative by Bria\'{n}ski, Davies, and Walczak~\cite{notPolyChi}. In this subsection, we discuss Esperet's Conjecture and how a theorem of Gir\~{a}o and Hunter~\cite{polyBoundedness} on degree-boundedness recovers some aspects of it.

Esperet's Conjecture said that, under a mild assumption about the class being closed under taking induced subgraphs, every $\chi$-bounded class has a very efficient $\chi$-bounding function. 

\begin{conjecture}[conjectured in~\cite{esperet2017habilitation}, disproven in~\cite{notPolyChi}]
Every hereditary, $\chi$-bounded class of graphs $\mathcal{F}$ has a $\chi$-bounding function that is a polynomial.
\end{conjecture}
\noindent We call a class of graphs $\mathcal{F}$ \emph{polynomially $\chi$-bounded} if it has a $\chi$-bounding function that is a polynomial. 

One reason Esperet's Conjecture garnered a lot of attention is that it intuitively seems so unlikely to be true. Yet the conjecture remained tantalizingly open for a number of years, despite many attempts at disproving it. Again and again, to the surprise of many people working in the area, the trend seemed to be in favor of Esperet's Conjecture. First it would be shown that a class of graphs is $\chi$-bounded. Then, often much later, it would be shown that the class is indeed polynomially $\chi$-bounded. Examples of this trend include classes of bounded rank-width~\cite{rankDecomp1, rankDecomp2}, classes of bounded twin-width~\cite{twinWidthIII, twinWidthColoring2, twinWidthColoring3}, circle graphs~\cite{gyarfas1985chromatic, kostochka1997covering, circleGraphsQuad, circleGraphsOpt}, grounded $L$-graphs~\cite{groundedLgraphs1, groundedLgraphs2}, broom-free graphs~\cite{kiersteadPenrice, BroomFree2}, and double-star-free graphs~\cite{kiersteadPenrice, DoubleStarFree2}, for instance. (Here, we list citations in order of improving $\chi$-bounding functions.)

Moreover, Esperet's Conjecture would have implied that every hereditary, $\chi$-bounded class of graphs $\mathcal{F}$ has the Erd\H{o}s–Hajnal property. A class $\mathcal{F}$ \emph{has the Erd\H{o}s–Hajnal property} if there exists $\epsilon>0$ so that every $n$-vertex graph in $\mathcal{F}$ has either a clique or an independent set of size at least $n^\epsilon$. Erd\H{o}s and Hajnal~\cite{EHConjecture} conjectured that every class of graphs which is hereditary and does not contain all graphs (up to isomorphism) satisfies this property. See the survey by Chudnovsky~\cite{EHsurvey} for a history of this conjecture. Let us just note that every such class does have much larger cliques or independent sets than normal~\cite{loglogEH, EHConjecture}, and that it is still open whether classes which forbid a path as an induced subgraph have the Erd\H{o}s-Hajnal property (even though it is known that such classes are $\chi$-bounded~\cite{GyarfasConjecture}, and there has been a lot of recent progress~\cite{EHP5, chiP5}, including a proof that the property ``almost'' holds~\cite{almostEHpaths}). Very recently, Nguyen, Scott, and Seymour~\cite{EHVCdimension} proved that any class of graphs of bounded VC-dimension has the Erd\H{o}s–Hajnal property. This theorem applies to a huge number of the graph classes which are known to be polynomially $\chi$-bounded. However, it does not apply to all of them -- bipartite graphs and their complements are simple examples.

With all of this motivation, it is unfortunate that Esperet's Conjecture is false. Also, not only is it false, but it fails in a very strong form. When they disproved Esperet's Conjecture, Bria\'{n}ski, Davies, and Walczak~\cite{notPolyChi} proved the following.

\begin{theorem}[\cite{notPolyChi}]
\label{thm:notPolyChi}
For any non-decreasing function $f:\mathbb{N} \rightarrow \mathbb{N}$ so that $f(1)=1$ and $f(n) \geq \binom{3n+1}{3}$ for any $n \geq 2$, there exists a hereditary class of graphs $\mathcal{F}$ whose optimal $\chi$-bounding function is~$f$.
\end{theorem}
\noindent That is, for each  $n \in \mathbb{N}$, the value $f(n)$ is equal to the maximum chromatic number of any graph in $\mathcal{F}$ with clique number at most $n$. So hereditary, $\chi$-bounded classes can have $\chi$-bounding functions that grow arbitrarily quickly; bounding the value of $f(n)$ does not impose any bound whatsoever on the value of $f(n+1)$. See also the generalized results in~\cite{isOfis}.

The construction used to prove Theorem~\ref{thm:notPolyChi} was inspired by a construction of Carbonero, Hompe, Moore, and Spirkl~\cite{inspoNotPolyChi} which was used to prove the following theorem. (A \emph{triangle} is just the graph $K_3$, so a graph is \emph{triangle-free} if it has no three pairwise adjacent vertices.)

\begin{theorem}[\cite{inspoNotPolyChi}]
\label{thm:triangleFreeChi}
There exists a hereditary class of graphs $\mathcal{F}$ so that $\mathcal{F}$ is not $\chi$-bounded and every triangle-free graph in $\mathcal{F}$ has chromatic number at most~$4$.
\end{theorem}

\noindent Informally, this theorem says that $f(2)$ does not ``control'' whether the class is $\chi$-bounded. We note that in the non-induced setting, it does; R\"{o}dl~\cite{R77} proved that for every integer $k$, every graph of sufficiently large chromatic number has a triangle-free subgraph (which is not necessarily induced) that has chromatic number at least $k$.

In the degree-boundedness setting, informally, it is actually true that $f(2)$ ``controls'' whether the class is degree-bounded. The corollary discussed in the last section, Corollary~\ref{cor:key}, immediately implies the following result.

\begin{corollary}
\label{cor:4-cycle-free}
A hereditary class of graphs $\mathcal{F}$ is degree-bounded if and only if there exists an integer $c$ so that every $4$-cycle-free graph in $\mathcal{F}$ has minimum degree at most $c$.
\end{corollary}
\noindent As discussed in the last section, this is a corollary of a theorem of Kwan, Letzter, Sudakov, and Tran~\cite{KLST20} combined with a theorem of the second author~\cite{inducedBip}.

In the case of $\chi$-boundedness, Theorem~\ref{thm:triangleFreeChi} ultimately lead to a disproof of Esperet's Conjecture. In a parallel chain of events, Corollary~\ref{cor:4-cycle-free} ultimately lead to a proof of the degree-boundedness version of Esperet's Conjecture. Gir{\~{a}}o and Hunter~\cite{polyBoundedness} proved the following.

\mainTheorem*
\noindent Similarly to before, we say a class of graphs $\mathcal{F}$ is \emph{polynomially degree-bounded} if $\mathcal{F}$ has a degree-bounding function that is a polynomial. We note that Theorem~\ref{thm:main} is essentially tight in that the degree of the polynomial cannot be uniformly bounded; there are examples where it must depend on the class $\mathcal{F}$. These examples essentially come from lower bounds on off-diagonal Ramsey numbers; see~\cite{singlyExponential}. 

Theorem~\ref{thm:main} culminated a long line of research. Already in the original proof of Corollary~\ref{cor:4-cycle-free} in~\cite{KLST20, inducedBip}, it was possible to obtain some bound on the rate of growth of an optimal degree-bounding function. (No bound was actually stated in~\cite{inducedBip}, but we believe that the proof yielded a bound that was roughly a tower of $2$'s of height $5$. This is not a particularly impressive bound, but it already would have shown much different behavior than $\chi$-boundedness, in light of Theorem~\ref{thm:notPolyChi}.) With Gir{\~{a}}o, Hunter, and Scott~\cite{singlyExponential}, the authors then proved that every hereditary degree-bounded class has a singly exponential degree-bounding function. Finally, in independent work that was concurrent with the work of Gir{\~{a}}o and Hunter, the researchers Bourneuf, Buci\'{c}, Cook, and Davies~\cite{polyBoundednessBip} reduced proving Theorem~\ref{thm:main} to the case of classes of bipartite graphs. We discuss these results in more detail in Section~\ref{sec:ratesOfGrowth}, focusing on the main ideas that go into proving Theorem~\ref{thm:main}.

\subsection{Minimum degree and clique number}
\label{subsec:cliqueNum}

There is an interesting case where Theorem~\ref{thm:main} actually does let us recover Esperet's Conjecture: hereditary classes where the minimum degree, rather than the chromatic number, is at most some function of the clique number.

So, let us say that a class of graphs $\mathcal{F}$ is \emph{degree-bounded by $\omega$} if there exists a function $f$ such that every graph $G \in \mathcal{F}$ satisfies $\delta(G) \leq f(\omega(G))$. There are many examples of graph classes which are degree-bounded by $\omega$. By Ramsey's Theorem, the class of all graphs with no independent set of size bigger than $\alpha$ is degree-bounded by $\omega$, for each fixed $\alpha \in \mathbb{N}$. For a similar reason, for each fixed star $S$, the class of all graphs with no induced subgraph isomorphic to $S$ is degree-bounded by $\omega$. In fact, it suffices to say that $V(G)$ can be ordered so that no vertex $v\in V(G)$ has an independent set of size bigger than $\alpha$ in its \emph{right} neighborhood. See~\cite{orderedGraphsChi} and the discussion in~\cite{polygonVisibility} for more information about ordered graphs and $\chi$-boundedness.

There are also many nice geometric examples of graph classes that are degree-bounded by $\omega$. For each integer $d$, the class of all intersection graphs of balls in $\mathbb{R}^d$ is degree-bounded by $\omega$. (The \emph{intersection graph} of a finite collection $\mathcal{B}$ of balls in $\mathbb{R}^d$ is the graph with vertex set $\mathcal{B}$, where two vertices are adjacent if they have non-empty intersection as balls in $\mathbb{R}^d$.) In fact, this holds if instead of balls we consider compact convex sets in $\mathbb{R}^d$ with bounded aspect ratio, following from work on separators~\cite{HarPeledQuanrud, MTTV, SmithWormald}. Dvo\v{r}\'{a}k, Norin, and the second author~\cite{DMN2021} also showed that we can allow the aspect ratio to be unbounded as long as for any pair of shapes $A,B \subseteq \mathbb{R}^d$ under consideration, we can ``smoothly slide one of $A,B$ around the interior of the other''. For instance, we can allow $A$ and $B$ to be axis-aligned rectangles as long as $B$ contains some translate of $A$ or vice-versa. We note that for all of these examples, it can actually be shown that the shape of smallest volume has small degree (as a function of $\omega$).


By combining Theorem~\ref{thm:main} of Gir\~{a}o and Hunter~\cite{polyBoundedness} with well-known bounds on off-diagonal Ramsey numbers~\cite{AKSRamsey}, we obtain the following result. Note that by considering a greedy/degeneracy coloring, Corollary~\ref{cor:degreeOmega} implies that every hereditary class which is degree-bounded by $\omega$ is also polynomially $\chi$-bounded. This is a special case of Esperet's Conjecture. 

\begin{corollary}[\cite{AKSRamsey, polyBoundedness}]
\label{cor:degreeOmega}
Let $\mathcal{F}$ be a hereditary class of graphs which is degree-bounded by $\omega$. Then there exists a polynomial $p$ so that every graph $G \in \mathcal{F}$ has $\delta(G) \leq p(\omega(G))$.
\end{corollary}
\begin{proof}
First of all, the class $\mathcal{F}$ is degree-bounded since every graph $G$ has $\tau(G)\geq \lfloor\omega(G)/2\rfloor$. Thus, by Theorem~\ref{thm:main}, there exists a polynomial $q$ so that every graph $G \in \mathcal{F}$ has $\delta(G) \leq q(\tau(G))$. We may assume that $q$ is non-decreasing. 

Next, notice that there exists a positive integer $s$ so that no graph in $\mathcal{F}$ has $K_{s,s}$ as an induced subgraph. By applying bounds on off-diagonal Ramsey numbers~\cite{AKSRamsey}, there thus exists $\epsilon>0$ so that every graph $G \in \mathcal{F}$ with biclique number $\tau = \tau(G)$ has a clique of size $\tau^\epsilon$. So, for every graph $G \in \mathcal{F}$, we have $\omega(G) \geq \tau(G)^{\epsilon}$, and thus $\delta(G) \leq q(\tau(G)) \leq q(\omega(G)^{1/\epsilon})$. So the polynomial $p(x)=q(x^{\lceil 1/\epsilon\rceil})$ has the desired property.
\end{proof}

\subsection{The Erd\H{o}s-Hajnal Conjecture} 

We already discussed the connection between Esperet's Conjecture and the Erd\H{o}s-Hajnal Conjecture in the last subsection. Now let us discuss the connection between the Erd\H{o}s-Hajnal Conjecture and Theorem~\ref{thm:main} about degree-boundedness. 

Given a graph $G$, we define the \emph{induced biclique number of $G$}, denoted $\widehat{\tau}(G)$, to be the maximum integer $s$ so that $G$ has an induced subgraph which is isomorphic to either the clique $K_{s}$ or the biclique $K_{s,s}$. In light of Theorem~\ref{thm:main}, we conjecture the following.

\begin{conjecture}
\label{conj:EHEquiv}
For any hereditary, degree-bounded class $\mathcal{F}$, there exists a polynomial $\widehat{p}$ so that every graph $G \in \mathcal{F}$ satisfies $\delta(G) \leq \widehat{p}(\widehat{\tau}(G))$.
\end{conjecture}

Let us note right away that Conjecture~\ref{conj:EHEquiv} is actually a special case of the Erd\H{o}s-Hajnal Conjecture and so was already conjectured in~\cite{EHConjecture}. In fact, we now use Theorem~\ref{thm:main} to argue that Conjecture~\ref{conj:EHEquiv} is equivalent to the Erd\H{o}s-Hajnal Conjecture for hereditary, degree-bounded classes.

\begin{lemma}
\label{lem:equivEH}
A hereditary, degree-bounded class $\mathcal{F}$ has the Erd\H{o}s-Hajnal property if and only if there exists a polynomial $\widehat{p}$ so that every graph $G \in \mathcal{F}$ satisfies $\delta(G) \leq \widehat{p}(\widehat{\tau}(G))$.
\end{lemma}
\begin{proof}
First, suppose that such a polynomial $\widehat{p}$ exists, and let $d$ be a number so that $\widehat{p}(x) = \mathcal{O}(x^d)$. Next, set $\epsilon = 1/(d+1)$, and suppose that $G$ is an $n$-vertex graph in $\mathcal{F}$ that does not have any clique or independent set of size at least $n^\epsilon$. Then the induced biclique number of $G$ is at most $n^\epsilon$, and thus $\delta(G) \leq \widehat{p}(n^\epsilon)=  \mathcal{O}(n^{\epsilon d})$. As every induced subgraph of $G$ also has induced biclique number at most $n^\epsilon$, in fact $G$ is $ \mathcal{O}(n^{\epsilon d})$-degenerate. Thus $G$ has an independent set of size at least $n/ \mathcal{O}(n^{\epsilon d}+1)$. 

Note that, in order to prove that $\mathcal{F}$ has the Erd\H{o}s-Hajnal property, it is enough to prove that every $n$-vertex graph in $\mathcal{F}$ has a clique or an independent set of size $\Omega(n^\epsilon)$. So we may assume that $n$ is sufficiently large, and thus
\begin{align*}
n/ \mathcal{O}(n^{\epsilon d}+1) \geq \Omega(n^{1-\epsilon d}) = \Omega(n^{1-d/(d+1)})=\Omega(n^{\epsilon}).
\end{align*}
So indeed, the class $\mathcal{F}$ has the Erd\H{o}s-Hajnal property.

In the other direction, suppose that $\mathcal{F}$ is a hereditary, degree-bounded class which has the Erd\H{o}s-Hajnal property. Let $\epsilon>0$ be such that every $n$-vertex graph $G \in \mathcal{F}$ has a clique or an independent set of size at least $n^\epsilon$. Then every graph $G \in \mathcal{F}$ with biclique number $\tau = \tau(G)$ has induced biclique number at least $\tau^{\epsilon}$, for we can transform each side of the biclique into either a clique or an independent set of size $\tau^{\epsilon}$. Thus $\widehat{\tau}(G) \geq \tau(G)^\epsilon$ for every $G\in \mathcal{F}$.

Now, since $\mathcal{F}$ is hereditary and degree-bounded, by Theorem~\ref{thm:main} there exists a polynomial $p$ so that every graph $G \in \mathcal{F}$ satisfies $\delta(G) \leq p(\tau(G))$. We may assume that $p$ is non-decreasing. Then, for every graph $G \in \mathcal{F}$, we have $\delta(G) \leq p(\tau(G)) \leq p(\widehat{\tau}(G)^{1/\epsilon})$. Thus we can take the polynomial $\widehat{p}(x) = p(x^{\lceil 1/\epsilon\rceil})$. This completes the proof of Lemma~\ref{lem:equivEH}.
\end{proof}

We recall from the earlier discussion that Nguyen, Scott, and Seymour~\cite{EHVCdimension} recently proved that every class of bounded VC-dimension has the Erd\H{o}s-Hajnal Property. This implies Conjecture~\ref{conj:EHEquiv} for many interesting cases. They also proved in~\cite{EHP5} that the class of graphs without an induced subgraph isomorphic to the $5$-vertex path has the Erd\H{o}s-Hajnal Property. This theorem resolved a case that had been open for a long time and received a lot of attention. However, it is still open whether the class of graphs with no induced subgraph isomorphic to the $6$-vertex path has the Erd\H{o}s-Hajnal Property, and this class is degree-bounded by~\cite{kiersteadPenrice}. So Conjecture~\ref{conj:EHEquiv} seems quite difficult. We note that Nguyen, Scott, and Seymour~\cite{almostEHpaths} recently proved that classes with a forbidden induced path ``almost'' have the Erd\H{o}s-Hajnal Property; they have cliques or independent sets of size $2^{(\log n)^{1-o(1)}}$.

\subsection{Direct implications}
\label{subsec:degNotChi}

In this section, we observe that neither $\chi$-boundedness nor degree-boundedness implies the other, even for hereditary classes. We also propose some direct connections that could exist. One direction of the claim is simple; the class of all bipartite graphs is $\chi$-bounded but not degree-bounded. The other direction, however, is more subtle. 

In his PhD thesis, Burling~\cite{burlingGraphs} constructed a class of graphs (now called Burling graphs) which are triangle-free and have arbitrarily large chromatic number. To begin with, the Burling graph $G_1$ is just the graph $K_2$, and one of its vertices is chosen to form a special independent set $I_1$ of $G_1$. Then, inductively, given $G_k$ and an independent set $I_k$ of $G_k$, we form $G_{k+1}$ and $I_{k+1}$ as follows. First, for each vertex $v \in I_k$, we replace $v$ by $|I_k|$-many copies of $v$, and then we glue on a copy $G_k^v$ of $G_k$ to these vertices by identifying them with the special independent set of $G_k^v$. Finally, for each $v \in I_k$ and each copy $x$ of $v$, we add two more vertices $x_1$ and $x_2$. We make these vertices $x_1$ and $x_2$ adjacent, and we make $x_1$ adjacent to all of the neighbors of $x$ which are in the original graph $G_k$, and we make $x_2$ adjacent to all of the neighbors of $x$ which are in the glued on copy $G_k^v$ of $G_k$. This graph $G_{k+1}$ is triangle-free, and its special independent set $I_{k+1}$ consists of all of the vertices $x$ and $x_1$ as defined above.

It can be shown that the chromatic number of the Burling graph $G_k$ is more than $k$. In fact something stronger holds; for any proper coloring of $G_k$ with any number of colors, there is a vertex in $I_k$ whose neighbors receive at least $k$ different colors. Suppose we have this for $G_k$ and wish to show it for $G_{k+1}$. The key point is that we can find a vertex $v \in I_k$ and a copy $x$ of $v$ so that at least $k$ different colors appear in the neighborhood of $x$ within 1) the original graph $G_k$ and 2) within the glued on copy $G_k^v$ of $G_k$. If these two sets of $k$ colors are different, then the vertex $x$ has the desired property. Otherwise, if they are the same, then the vertex $x_1$ has the desired property since $x_2$ receives a new $(k+1)$th color.

Burling graphs are often a source of counterexamples on $\chi$-boundedness, and this turns out to be the case here as well.

\begin{theorem}[Corollary of~\cite{burlingGraphs, foxPach2010, representBurling}]
\label{thm:degNotChi}
The class of all graphs $G$ which are an induced subgraph of some Burling graph is degree-bounded but not $\chi$-bounded. 
\end{theorem}

We already argued that this class is not $\chi$-bounded. To see that it is degree-bounded, we combine two theorems. First, in order to disprove another conjecture about $\chi$-boundedness, Pawlik, Kozik, Krawczyk, Laso\'{n}, Micek, Trotter, and Walczak~\cite{representBurling} proved that Burling graphs can be represented by the intersection pattern of a collection of line segments in the plane. That is, they proved that for any Burling graph $G_k$, there is a collection $\{L_v: v \in V(G_k)\}$ of line-segments in the real plane $\mathbb{R}^2$ so that for any pair of distinct vertices $u,v \in V(G_k)$, the vertices $u$ and $v$ are adjacent in $G_k$ if and only if the line segments $L_v$ and $L_u$ intersect. Second, Fox and Pach~\cite{foxPach2010} proved that the class of all graphs with such a representation is degree-bounded. (In fact, they proved this even for the more general class of string graphs which will be discussed in Subsection~\ref{subsec:geometric}.) Since the class of all graphs with such a representation is hereditary, Theorem~\ref{thm:degNotChi} follows.

If we wish to save something in this direction, it is natural to ask for the maximum chromatic number of an $n$-vertex triangle-free graph in a fixed degree-bounded class. Here we are aiming for a bound that is relatively small in $n$ rather than constant. In general, every $n$-vertex triangle-free graph has chromatic number at most $(2+o(1))\sqrt{n/\log(n)}$~\cite{chiRamseyUpper}, and there exist $n$-vertex triangle-free graphs with chromatic number at least $(1/\sqrt{2}-o(1))\sqrt{n/\log(n)}$~\cite{chiRamseyLower1, chiRamseyLower2}. It seems possible that in degree-bounded classes, the answer is significantly lower.

\begin{problem}
\label{prob:chiAndDeg}
Let $\mathcal{F}$ be a hereditary, degree-bounded class. Does every $n$-vertex, triangle-free graph in $\mathcal{F}$ have chromatic number at most $n^{o_{\mathcal{F}}(1)}$? What is the best bound?
\end{problem}
\noindent We write $o_{\mathcal{F}}(1)$ for a function whose limit as $n$ goes to infinity is $0$, when the class $\mathcal{F}$ is fixed. We note that Problem~\ref{prob:chiAndDeg} equivalently asks if every $n$-vertex, triangle-free graph in $\mathcal{F}$ has an independent set of size at least $n^{1-o_{\mathcal{F}}(1)}$. (Since the class $\mathcal{F}$ is hereditary, the desired coloring can then be obtained by greedily removing a maximum independent set.)

In upcoming work, Nguyen, Scott, and Seymour~\cite{almostLinearTrees} prove that for each fixed tree $H$, every $n$-vertex, triangle-free graph with no induced subgraph isomorphic to $H$ has chromatic number at most $n^{o_H(1)}$. The intersection graphs mentioned before -- string graphs -- satisfy an even stronger bound. Generalizing work of McGuinness~\cite{stringsBoundedIntersectionsChi}, Fox and Pach~\cite{stringGraphsChi} showed that every $n$-vertex, triangle-free string graph has chromatic number at most polylogarithmic in $n$. In upcoming work, Nguyen, Scott, and Seymour~\cite{almostLinearSubdivision} generalize this theorem to graphs with a forbidden induced subdivision of any fixed graph $H$. All of these results also hold for graphs which are $K_t$-free, rather than triangle-free. Finally, we note that there is a long history of studying the maximum chromatic number of $n$-vertex, triangle-free graphs with geometric representations; see also~\cite{betterChiStringTriangleFree, HasseDiagramsChiLarge, Lshapesloglog} for other results.

In another direction, we note that maybe the chromatic number is still a ``reasonably local'' property in any hereditary degree-bounded class. Let us be precise. First of all, a class of graphs $\mathcal{F}$ is \emph{$r$-controlled} if there exists a function which bounds the chromatic number of each graph $G \in \mathcal{F}$ in terms of the maximum chromatic number of any of its induced subgraphs of radius at most $r$. (The \emph{radius} of a graph $H$ is the minimum over all vertices $v \in V(H)$, of the maximum distance between $v$ and any other vertex of $H$. The \emph{distance} between two vertices is the minimum number of edges in any path that joins them.) Thus a hereditary class is $\chi$-bounded if and only if it is $1$-controlled. While the class of all string graphs is not $1$-controlled~\cite{representBurling}, it is $2$-controlled~\cite{chandeliersStrings}. It would be nice to have an example of the following.

\begin{problem}
\label{prob:rControlled}
Find a hereditary graph class $\mathcal{F}$ which is degree-bounded but is not $r$-controlled for any integer $r$.
\end{problem}

\subsection{Perfect graphs}
\label{subsec:perfect}

A graph $G$ is \emph{perfect} if for every induced subgraph $H$ of $G$, the chromatic number of $H$ is equal to its clique number, that is, $\chi(H) = \omega(H)$. Berge~\cite{bergePerfect} conjectured a wonderful characterization of perfect graphs: that a graph is perfect if and only if it does not have an induced subgraph which is a cycle of length odd and greater than $3$, or the complement of such a graph. (The \emph{length} of a cycle or a path is the number of edges it contains.) Chudnovsky, Robertson, Seymour, and Thomas~\cite{strongPerfect} proved this conjecture in 2006, for which they won the famous Fulkerson Prize in discrete mathematics. Their theorem is known as the Strong Perfect Graph Theorem; the Weak Perfect Graph Theorem is due to Lov\'{a}sz~\cite{weakPerfect} and says that a graph is perfect if and only if its complement is perfect. Perfect graphs also have strong connections to combinatorial optimization. They can be used to characterize when the polytope of a set packing problem has only integral vertices; see the book by Cornu\'{e}jols~\cite{perfectMatrices}.


In our setting, it is natural to propose the following analog of perfect graphs. We say that a graph $G$ is \emph{degree-perfect} if for every induced subgraph $H$ of $G$, the minimum degree of $H$ is at most its biclique number, that is, $\delta(H) \leq \tau(H)$. Together with the standard characterization of bipartite graphs, a theorem of Golumbic and Goss~\cite[Theorem~4]{bipChordal} implies the following characterization of degree-perfect graphs.

\begin{theorem}[Corollary of~\cite{bipChordal}]
\label{thm:degreePerfect}
A graph is degree-perfect if and only if it does not have an induced subgraph which is a cycle of length other than four.
\end{theorem}

In some ways, the fact that a triangle is not degree-perfect and is thus forbidden in Theorem~\ref{thm:degreePerfect} seems artificial. Perhaps another interesting characterization could be obtained if, in the definition of degree-perfect graphs, we replaced $\tau(G)$ by $\max(\tau(G), \omega(G)-1)$, or even by the maximum minimum degree of a complete multipartite subgraph. Under both of these definitions cliques would become degree-perfect.

We note that the class of graphs considered in Theorem~\ref{thm:degreePerfect} is of independent interest; they are the ``bipartite analog'' of chordal graphs. Recall that a graph is \emph{chordal} if its only induced cycles are copies of the triangle $C_3$. Theorem~\ref{thm:degreePerfect} equivalently says that degree-perfect graphs are precisely the bipartite graphs whose only induced cycles are copies of the $4$-cycle $C_4$. The theorem of Golumbic and Goss~\cite[Theorem~4]{bipChordal} which implies Theorem~\ref{thm:degreePerfect} says that such graphs have simplicial edges. Simplicial edges are defined analogously to the simplicial vertices which are studied in chordal graphs; an edge $e=\{u,v\}$ of a bipartite graph $G$ is \emph{simplicial} if $G$ contains all possible edges between $N(u)$ and $N(v)$. See~\cite{bipChordal} for details.

\subsection{Induced subdivisions}
\label{subsec:inducedSub}

We say that a graph $G$ contains a graph $H$ as an \emph{induced subdivision} if there exists an induced subgraph of $G$ which is isomorphic to a subdivision of $H$. A \emph{subdivision} of $H$ is any graph which can be obtained from $H$ by replacing each edge $e = \{u,v\}$ of $H$ by a path $P_e$ whose ends are $u$ and $v$ and whose internal vertices are all distinct, new vertices. For each graph $H$, let us write $\mathcal{F}_H$ for the class of all graphs which do not contain $H$ as an induced subdivision.

It was originally conjectured that for every graph $H$, the class $\mathcal{F}_H$ is $\chi$-bounded~\cite{subdivisionTreesChi}. This problem has some nice motivation. First of all, Scott~\cite{subdivisionTreesChi} proved that it is true when $H$ is a tree (that is, that the class $\mathcal{F}_T$ is $\chi$-bounded for any tree $T$). Moreover, one of Gy\'{a}rf\'{a}s'~\cite{GyarfasConjecture} original problems on $\chi$-boundedness from the 80's equivalently says that the class $\mathcal{F}_C$ is $\chi$-bounded for any cycle $C$. This conjecture was later proven by Chudnovsky, Scott, and Seymour~\cite{longHolesChiBd}. (In fact, with Spirkl~\cite{longOddHoles}, they also showed that it is enough to forbid all induced odd cycles of length at least $\ell$, for any fixed integer $\ell$. A cycle is \emph{odd} if it has an odd number of edges.)

Overall, however, Pawlik et. al.~\cite{representBurling} proved that this conjecture is false by using the same construction from Theorem~\ref{thm:degNotChi}: Burling graphs. This construction shows that, if $K_5^1$ denotes the $1$-subdivision of $K_5$, then the class $\mathcal{F}_{K_5^1}$ is not $\chi$-bounded. (The \emph{1-subdivision} of a graph $H$ is the subdivision of $H$ where for every edge $e=\{u,v\}$ of $H$, the path $P_e$ which replaces $e$ contains exactly one additional vertex besides $u$ and $v$.) In fact this construction forbids induced subdivisions of even more graphs; see~\cite{restrictedFrameGraphs}. Other positive results have also been proved, in particular for graphs $H$ which are ``banana trees''; see~\cite{chandeliersStrings} and~\cite{bananaTrees}.

In the case of degree-boundedness, however, the analogous conjecture is actually true. K\"{u}hn and Osthus~\cite{KO04induced} proved a wonderful theorem which equivalently says the following.

\begin{theorem}[\cite{KO04induced}]
\label{thm:inducedSubdivision}For any graph $H$, the class $\mathcal{F}_H$ of all graphs with no induced subdivision of $H$ is degree-bounded.
\end{theorem}

\noindent Dvo\v{r}\'{a}k~\cite{D18} proved that such classes actually satisfy a much stronger property. He showed that for any graph $H$, any subclass of $\mathcal{F}_H$ with bounded biclique number has not only bounded minimum degree $\delta$, but in fact also has bounded expansion. We discuss this further in Section~\ref{subsec:beyond}.

Let us also point out that there is a long history of improving the bounds in Theorem~\ref{thm:inducedSubdivision}, that is, of finding efficient degree-bounding functions for classes with a forbidden induced subdivision. Bonamy, Bousquet, Pilipczuk, Rz\polhk{a}\.{z}ewski, Thomass\'{e}, and Walczak~\cite{polyBoundPath} proved that for any path $P$ or cycle $C$, the classes $\mathcal{F}_P$ and $\mathcal{F}_C$ are polynomially degree-bounded. They also conjectured that the class $\mathcal{F}_H$ is polynomially degree-bounded for every graph $H$. This conjecture was proven independently by Bourneuf, Buci\'{c}, Cook, and Davies~\cite{polyBoundednessBip}, and by Gir{\~{a}}o and Hunter~\cite{polyBoundedness}. Finally, Gir{\~{a}}o and Hunter~\cite{polyBoundedness} proved the full conjecture in an even more general form by proving Theorem~\ref{thm:main}. Overall, the conjecture on induced subdivisions has motivated much of the work on degree-boundedness.

We would also like to point the reader towards the following conjecture of Scott and Seymour~\cite{ss20survey}, which combines some of the topics discussed here. We say that a subdivision of a graph $H$ is \emph{odd} (respectively, \emph{even}) if for each edge $e \in E(H)$, the path $P_e$ used to replace $e$ contains an odd (respectively, even) number of edges.

\begin{conjecture}[\cite{ss20survey}]
\label{conj:oddSubdivision}
For any graph $H$ and integer $t$, the class of all graphs with no subgraph isomorphic to $K_{t,t}$ and no induced odd subdivision of $H$ has bounded chromatic number.
\end{conjecture}
\noindent We refer the reader to~\cite{ss20survey} for a discussion on the progress towards Conjecture~\ref{conj:oddSubdivision}.

We note that for any graph $H$ which contains an odd cycle, the class of all graphs with no induced odd subdivision of $H$ is not degree bounded, because it contains all bipartite graphs. For even subdivisions, however, we obtain the following result as a corollary of Theorem~\ref{thm:inducedSubdivision} and Theorem~\ref{thm:KLST}.

\begin{corollary}
For any graph $H$, the class of all graphs with no induced even subdivision of $H$ is degree-bounded.
\end{corollary}
\begin{proof}
By Theorem~\ref{thm:KLST}, it suffices to prove that the class of all \emph{bipartite} graphs with no induced even subdivision of $H$ is degree-bounded. We claim that this class also excludes all induced subdivisions of a clique on $(2|V(H)|-1)$-many vertices. So, suppose that $G$ is a bipartite graph which has an induced subdivision of a clique on $(2|V(H)|-1)$-many vertices. Note that for any two vertices $u$ and $v$ in the same side of $G$, every path between $u$ and $v$ has an even number of edges. So by the pigeonhole principle, $G$ contains an induced even subdivision of a clique on $|V(H)|$-many vertices, and thus an induced even subdivision of $H$. 

The corollary now follows from Theorem~\ref{thm:inducedSubdivision}.
\end{proof}

\subsection{Induced trees}
\label{subsec:tree}
The most famous problem on $\chi$-boundedness which is still open is the  Gy\'{a}rf\'{a}s-Sumner Conjecture~\cite{GSConjecture1, GSConjecture2}, which proposes that for any forest $T$, the class of all graphs $G$ with no induced subgraph isomorphic to $T$ is $\chi$-bounded. (It is known that forests are the only graphs that can have this property, as there exist graphs of arbitrarily large girth and chromatic number; see for instance the classical random construction of Erd\H{o}s~\cite{ErdosLargeChi}.) There is a lot of partial progress on this conjecture; see~\cite{ss20survey}. Notably, Gy\'{a}rf\'{a}s~\cite[Theorem~2.4]{GyarfasConjecture} proved the conjecture for paths using a now famous path-augmentation argument. 

While the Gy\'{a}rf\'{a}s-Sumner Conjecture remains open, its degree-bounded analog has been known for a long time, as discussed in the survey of Scott and Seymour~\cite{ss20survey}. First of all, Hajnal and R\"{o}dl independently proved that for any forest $T$ and any integer $t$, the class of all graphs $G$ with no induced subgraph isomorphic to $T$ and with $\tau(G) \leq t$ is $\chi$-bounded; see~\cite{treesAndChi}. Kierstead and Penrice~\cite[Corollary 7]{kiersteadPenrice} then proved a theorem which implies the following.

\begin{theorem}[\cite{kiersteadPenrice}] For any forest $T$, the class of all graphs with no induced subgraph isomorphic to $T$ is degree-bounded.
\end{theorem}

Then, motivated by a question of Bonamy, Bousquet, Pilipczuk, Rz\polhk{a}\.{z}ewski, Thomass\'{e}, and Walczak~\cite{polyBoundPath} discussed in the last subsection, Scott, Seymour, and Spirkl~\cite{polyBoundTree} gave a clever argument which shows a polynomial degree-bound. That is, they showed that for any forest $T$, the class of all graphs with no induced subgraph isomorphic to $T$ is polynomially degree-bounded. Hunter, Milojevi\'{c}, Sudakov, and Tomon~\cite{KSTHereditaryConj} further improved this result by showing that there is a degree-bounding function whose degree is linear in the number of vertices of $T$. So for trees, things are rather well-understood in the setting of degree-boundedness.


\section{Degree-bounding functions}
\label{sec:ratesOfGrowth}


In this section, we discuss the rate of growth of optimal degree-bounding functions. As mentioned before, in contrast to $\chi$-boundedness, Gir\~{a}o and Hunter~\cite{polyBoundedness} proved that every hereditary degree-bounded class has a degree-bounding function that is a polynomial. In this section we give a broad overview of how to prove this sort of theorem. 

The key idea is that the minimum degree of the $4$-cycle free graphs in the class ``controls'' the rate of growth of the whole function. Throughout this section, we let $\mathcal{F}$ denote an arbitrary hereditary  degree-bounded class. Recall from Lemma~\ref{lem:minDegreeAvgDeg} that bounding the minimum degree is the same thing as bounding the average degree. Thus, there exists an integer $k:=k({\mathcal{F}})$ such that every $4$-cycle-free graph in $\mathcal{F}$ has average degree less than $k$. The authors, together with Gir\~{a}o, Hunter, and Scott~\cite{singlyExponential}, proved that every hereditary degree-bounded class has a degree-bounding function that is singly exponential in the biclique number~$\tau$. 
 
\begin{theorem}[\cite{singlyExponential}]\label{singleexponential}
    There exists a universal constant $C$ so that $f_k(\tau)=k^{C\tau^3}$ is a degree-bounding function for any hereditary degree-bounded class $\mathcal{F}$.
\end{theorem} 
\noindent The broad-level approach used to prove Theorem~\ref{singleexponential} remains the same in the later proofs which obtained better bounds. So let us briefly discuss the key improvements.
 
Bourneuf, Buci\'{c}, Cook, and Davies~\cite{polyBoundednessBip} reduced polynomial degree-boundedness to the case of bipartite graphs, and they proved a wonderful strengthening of the K\H{o}v\'{a}ri-S\'{o}s-Tur\'{a}n Theorem~\cite{KST}. The K\H{o}v\'{a}ri-S\'{o}s-Tur\'{a}n Theorem is a key tool in all known proofs, and this improvement is a key advancement of Bourneuf, Buci\'{c}, Cook, and Davies. Simultaneously, Gir\~{a}o and Hunter~\cite{polyBoundedness} independently discovered the same key tool. They used this tool, along with some wonderful lemmas they proved about VC-dimension of a similar flavor to the Sauer-Shelah Lemma~\cite{Sauer, Shelah}, to prove the following degree-bounding function that is polynomial in $\tau$. 

\begin{theorem}[\cite{polyBoundedness}]\label{polynomial}
   There exists a universal constant $C$ so that $p_k(\tau)=\tau^{Ck^4}$ is a degree-bounding function for any hereditary degree-bounded class $\mathcal{F}$.
\end{theorem} 
\noindent In the next few subsections, we discuss the main parts of the proof at a high level. 

\subsection{Dichotomy}
\label{subsec:dichotomy}

Aiming for a contradiction, suppose that there is a graph $G$ in $\mathcal{F}$ with biclique number $\tau$ whose average degree exceeds the desired degree bound. If we are aiming for the degree bound of $d$, then we may assume that the graph $G$ is roughly $d$-degenerate, since otherwise we could instead consider a smaller induced subgraph of $G$ whose average degree still exceeds the bound.
The proofs of all three mentioned results find a $4$-cycle-free induced subgraph of $G$ with average degree at least $k$, which leads to a contradiction as $\mathcal{F}$ is hereditary. Given that $G$ has large average degree, the problem can be separated into the following two cases:
\begin{itemize}
    \item[(A)] $G$ contains an induced subgraph that is ``almost regular'' and has large degree, or
    \item[(B)] $G$ contains an induced bipartite subgraph with sides $A$ and $B$ such that $|A|$ is much greater than $|B|$ and every vertex in $A$ has large degree.
\end{itemize}

A similar approach was already used earlier by Kwan, Letzter, Sudakov, and Tran~\cite{KLST20}. One of the improvements made in~\cite{singlyExponential} was to use an approach of Montgomery, Pokrovskiy, and Sudakov~\cite{MPS2021} to ``bump'' how large $A$ is in comparison to $B$. So there is an intermediate step where a ``biregular'' graph of large average degree is found; this in turn guarantees the existence of an ``almost regular'' induced subgraph of large degree; see~\cite[Lemma~3.7]{ErdosSauer} in a paper of Janzer and Sudakov which resolved the {E}rd\H{o}s-{S}auer problem.

Since this type of dichotomy is (at least at a broad level) a common element of all of the approaches, we focus instead on the key new improvements of Bourneuf, Buci\'{c}, Cook, and Davies~\cite{polyBoundednessBip}, and of Gir\~{a}o and Hunter~\cite{polyBoundedness} which were used to take care of the two cases. Subsections~\ref{subsec:almostRegular} and~\ref{subsec:unbalBip} are devoted to introducing the techniques used to handle the two respective cases. Here is a more rigorous statement of the type of dichotomy theorem that is often used. Note that in case~(B) of the theorem below, the graph is not bipartite yet. This can be dealt with later on using degeneracy.

\begin{theorem}[{\cite[Lemma 3.4]{polyBoundedness}}]
\label{dichotomy}
    Let $L,d\ge 16$ be positive integers. Then for any $n$-vertex graph $G$ with average degree $d$ that is $d$-degenerate, either:
    \begin{itemize}
        \item[(A)] $G$ has an induced subgraph $G_0$ with average degree $d_0$ at least $\Omega(d/\log^2(L))$ and maximum degree $\Delta_0$ at most $O(d_0\cdot\log^2(L))$, or
        \item[(B)] the vertex set of $G$ admits a partition $V(G)=A\sqcup B$ such that $|A|\ge L|B|$ and the number of edges between $A$ and $B$ is at least $nd/8.$
    \end{itemize}
\end{theorem}


\subsection{Almost regular graphs}
\label{subsec:almostRegular}

In all of the mentioned proofs, case~(A) is resolved using variants of the K\H{o}v\'{a}ri-S\'{o}s-Tur\'{a}n Theorem. Recall the following statement mentioned in Section 2.2.

\strongerKST*

In fact, Gir\~{a}o and Hunter~\cite[Lemma~7.1]{polyBoundedness} proved that when $t$ is sufficiently large in comparison to the number of vertices of $H$, we can take $\epsilon_H = 1/(100\cdot \Delta(H))$. The following lemma, also proven by Gir\~{a}o and Hunter, gives a good choice of the graph $H$. For any integers $\ell$ and $r$ with $r \leq \ell$, we write $K_\ell^{(r)}$ for the bipartite graph with $\ell$ vertices on one side and $\binom{\ell}{r}$ vertices on the other side, where for each $r$-subset of the $\ell$ vertices, there exists a vertex whose neighborhood is precisely that $r$-set. (An \emph{$r$-set} is just a set of size $r$.)


\begin{lemma}[{\cite[Lemma~7.4]{polyBoundedness}}]
    For each integer $k\ge 2$, there exists a $4$-cycle-free bipartite graph $H = H(k)$ which has at most $8k^2$ vertices, has maximum degree at most $2k$, has average degree at least $k$, and is an induced subgraph of $K_{8k^2}^{(2k)}$.
\end{lemma}

We may assume that $k \geq 2$. So let us now fix a graph $H$ as above. Let $\epsilon_H$ and $p_H(t)$ be as in Theorem \ref{thm:strongerKST}. In fact, as mentioned before, we can take $\epsilon_H = 1/(100\cdot \Delta(H)) \geq 1/(200k)$. So $\epsilon_H$ is even relatively large. Note that $H$ is not an induced subgraph of any graph in $\mathcal{F}$, and so we may apply Theorem~\ref{thm:strongerKST} to such graphs. We suppose that case~(A) of the dichotomy theorem (Theorem~\ref{dichotomy}) holds. Thus, $G_0$ is a graph in $\mathcal{F}$ which has average degree $d_0$ at least $\Omega(d/\log^2(L))$ and maximum degree $\Delta_0$ at most $O(d_0\cdot\log^2(L))$. We let $t$ be the smallest integer so that $G_0$ has no subgraph isomorphic to $K_{t,t}$. When this dichotomy theorem is applied, we will have that $\log^2(L)$ is still fairly small in comparison to $d$. So the maximum degree will not be too much larger than the average degree.

The proof proceeds by finding an induced subgraph of $G_0$ by including vertices independently at random with some probability $q$ and then deleting the vertices in any $4$-cycle. The resulting graph will be $4$-cycle-free and will, in expectation, have average degree more than $k$. In this manner we obtain a contradiction. The key observation is that Theorem~\ref{thm:strongerKST} can be used to obtain a rather efficient bound on the number of $4$-cycles. The point is that for every edge $xy\in E(G_0)$, by applying Theorem \ref{thm:strongerKST} to the subgraph of $G_0$ induced on the neighbors of $x$ and $y$, we know there are at most $8p_H(t)\Delta_0^{2-\epsilon_H}$-many 4-cycles containing the edge $xy$ in $G_0$. (This induced subgraph has at most $2\Delta_0$-many vertices, and each of its edges contributes at most two such $4$-cycles.) This is the main idea and turns out to be enough for the proof to go through.

\subsection{Unbalanced bipartite graphs}
\label{subsec:unbalBip}

Suppose now that case~(B) in Theorem \ref{dichotomy} holds, that is, there exists a partition $V(G)=A\sqcup B$ such that $|A|\ge L|B|$ and the number of edges between $A$ and $B$ is at least $nd/8$.

An argument using the $d$-degeneracy of $G$ shows that at least $|A|/20$ vertices in $A$ have at least $d/20$ and at most $10d$ neighbors in $B$; among those vertices, since $G$ is $d$-degenerate, there exists an independent subset $A_0$ of size  $|A_0|\ge (|A|/20)/(d+1)\ge |A|/40d$.
We then pass to the induced subgraph $G_0$ of $G$ which has vertex-set $A_0\cup B$, where $|A_0|\geq \frac{L}{40d}|B|$,  $A_0$ is independent in $G_0$, and every vertex in $A_0$ has degree between $d/20$ and $10d$ in the graph~$G_0$.

Before continuing with the proof sketch of the polynomial bound, we briefly mention the different technique used in the proof of the singly exponential bound (Theorem \ref{singleexponential}), and reveal the  choice of $L$.
The proof of Theorem \ref{singleexponential} applies the dichotomy with $L=2^{d^\delta}$, where $\delta=1/200t$ and $t$ is the smallest integer such that $G$ has no subgraph isomorphic to $K_{t,t}$. A probabilistic argument similar to the one in Subsection~\ref{subsec:almostRegular} resolves case~(A) by applying the classical extremal bound from the K\H{o}v\'{a}ri-S\'{o}s-Tur\'{a}n Theorem~\cite[Lemma 3.3]{singlyExponential}. For case~(B), the proof is done by first passing to an induced bipartite subgraph $G[A'\cup B']$ where $A'\subseteq A,B'\subseteq B$ are independent and every vertex in $A'$ has the same degree. Then the proof proceeds by applying a theorem of F\"uredi~\cite{Furedi83} about uniform hypergraphs to find an induced $4$-cycle-free subgraph with average degree at least $k$~\cite[Theorem 5.1]{singlyExponential}. 

In the proof of Theorem \ref{polynomial}, the dichotomy is instead performed with $L:=d^{D+3}$, where $D$ is an integer such that every $K_{k,k}$-free graph in $\mathcal{F}$ has average degree less than $D$, which exists because $\mathcal{F}$ is degree-bounded.
The proof sketch in Subsection~\ref{subsec:almostRegular} for case~(A) goes through as $d_0=\Omega(d/\log^2(L))\geq \Omega(d/(D\log{d})^2)$ and $\Delta/d_0=O(\log^2(L))=O((D\log(d))^2)\le {O}(d_0^{1/1000k})$, when $d$ is large enough.


We say that a collection $\mathcal{S}$ of subsets of some ground set, in our case the set of vertices $B$, \emph{$k$-shatters} a set $R\subseteq B$ if for every $k$-subset $S\subseteq R$, there exists a set in $\mathcal{S}$ whose intersection with $R$ is precisely $S$. Given the graph $G_0$, consider a vertex $x \in A_0$. For any independent subset $J$ contained in the neighborhood of $x$, if the collection $\{N(x'):x'\in A_0\setminus\{x\}\}$ $2k$-shatters an $8k^2$-subset of $J$, then $G_0$ contains the graph $K_{8k^2}^{(2k)}$ as desired.
 
The rest of the proof utilizes the following ``shattering lemma". The idea is to apply this lemma within independent sets contained in the neighborhood of vertices $x \in A_0$. If the lemma does not apply, then instead we find an induced subgraph where certain pairs of vertices have few common neighbors. This will also let us remove all big bicliques and find another win. 

\begin{lemma}[{\cite[Lemma 8.4]{polyBoundedness}}]\label{shattering}
    Let $D,t\ge 2$ and $h$ be integers. Then there exists $\delta>0$ such that the following holds when $n$ is large enough. For any collection $\mathcal{S}$ of subsets of $\{1,2,\ldots, n\}$ with $|F|<\delta n$ for all $F\in\mathcal{S}$, if $\mathcal{S}$ does not $t$-shatter an $h$-subset of $\{1,2,\ldots, n\}$, then there exists $I\subseteq \{1,2,\ldots, n\}$ of size $D$ such that $|I\cap F|<t$ for all $F\in \mathcal{S}$.
\end{lemma}

By arguments similar to case~(A), one can prove that  we can find large independent sets in the neighborhood of each vertex. The proof is finished by using a probabilistic argument to find an induced subgraph which is $K_{k,k}$-free yet still has average degree at least $D$. This completes the broad overview of some techniques used to obtain a polynomial degree-bounding function.





\section{Examples}
\label{sec:example}

We have already seen a number of examples of degree-bounded classes in Section~\ref{sec:ChiBoundedness}, where we discussed connections to $\chi$-boundedness. These examples included:\begin{itemize}
    \item classes with bounded independence number~\cite{AKSRamsey} (see Subsection~\ref{subsec:cliqueNum}),
    \item intersection graphs of compact convex sets in $\mathbb{R}^d$ of bounded aspect ratio~\cite{HarPeledQuanrud, MTTV, SmithWormald}, and some of their generalizations~\cite{DMN2021} (see Subsection~\ref{subsec:cliqueNum}), 
    \item intersection graphs of line segments in the plane~\cite{foxPach2010} (see Subsection~\ref{subsec:degNotChi}),
    \item the bipartite analog of chordal graphs~\cite{bipChordal} (see Subsection~\ref{subsec:perfect}),
    \item classes which forbid induced subdivisions of a fixed graph~\cite{KO04induced} (see Subsection~\ref{subsec:inducedSub}), and
    \item classes which forbid a fixed tree as an induced subgraph~\cite{kiersteadPenrice} (see Subsection~\ref{subsec:tree}).
\end{itemize}

In this section we focus on two types of examples: classes with nice geometric representations, and classes with bounded width. Many examples of geometrically-defined classes of graphs that are degree-bounded come from intersection graphs, and those are the ones that we focus on here. We note that, additionally, it is open whether the class of curve visibility graphs is degree-bounded; see~\cite{curveVisibility} for the definition of this class. The subclass of terrain visibility graphs is degree-bounded due to connections with twin-width; see~\cite{delineationTwinWidth}. The case of visibility graphs is particularly interesting because they contain an induced subdivision of every graph, but there are many more geometric examples one could consider.

We begin by stating the general definitions for intersection graphs. Given a collection $\mathcal{B}$ of subsets of some ground set $S$, the \emph{intersection graph of $\mathcal{B}$} is the graph with vertex set $\mathcal{B}$ where two vertices are adjacent if they have non-empty intersection as subsets of $S$. We write $G(\mathcal{B})$ for the intersection graph of $\mathcal{B}$. We generally require that $\mathcal{B}$ -- and thus the graph $G(\mathcal{B})$ -- is finite, but we impose no restrictions on $S$.

We note that every graph $G$ is an intersection graph of a collection of subsets of $E(G)$. For each vertex $v$ of $G$, let $B_v\subseteq E(G)$ be the set of all edges of $G$ that are incident to $v$. Then two vertices $u$ and $v$ are adjacent in $G$ if and only if $B_u$ and $B_v$ intersect. (We say that two sets \emph{intersect} if their intersection is non-empty.) So in order to obtain a class of graphs that is degree-bounded, we need to impose some structure on how the sets in $\mathcal{B}$ can intersect. We discuss a number of ways to use geometry to impose this structure.


\subsection{Convex sets in Euclidean space}
\label{subsec:geometric}

For our first example of geometric intersection graphs, let the ground set be $S=\mathbb{R}^d$, and let $\mathcal{B}$ be any collection of balls in $\mathbb{R}^d$. That is, for each element $B \in \mathcal{B}$, there is a center point $p \in \mathbb{R}^d$ and a radius $r \in \mathbb{R}_{>0}$ so that $B = \{x \in \mathbb{R}^d: \lVert x -p\rVert_2 \leq r\}$. A graph $G$ is an \emph{intersection graph of balls in $\mathbb{R}^d$} if there exists such a collection $\mathcal{B}$ so that $G$ is isomorphic to the intersection graph $G(\mathcal{B})$. Characterizing these graphs and determining how dense they can be has a long and fascinating history. For one rich example, we refer the reader to the expository article by Cohn~\cite{spherePacking} for more information about the density of sphere packings. 

As for the characterization problem, there is a famous theorem of Koebe~\cite{Koebe} (also discovered independently by Andreev~\cite{Andreev} and Thurston~\cite{Thurston}) which characterizes the intersection graphs of internally disjoint balls in the plane. A collection $\mathcal{B}$ of balls in $\mathbb{R}^d$ is \emph{internally disjoint} if the interiors of the balls in $\mathcal{B}$ are pairwise disjoint.

\begin{circlePacking}[\cite{Andreev, Koebe, Thurston}] A graph $G$ is the intersection graph of a collection of internally disjoint balls in the plane if and only if $G$ is planar.
\end{circlePacking}
\noindent We also refer the reader to the proof of Brightwell and Scheinerman~\cite{niceCirclePacking}, which is stated in accessible language and finds a simultaneous representation of the planar dual.

It is a well-known corollary of Euler's formula for planar graphs that every planar graph has a vertex of degree at most $5$. This type of formula can be roughly generalized to arbitrary intersection graphs of balls in $\mathbb{R}^d$ as follows. First, we say that a collection $\mathcal{B}$ of subsets of a ground set $S$ is \emph{$k$-thin} if every point in $S$ is in at most $k$ elements of $\mathcal{B}$. Note that if the clique number of $G(\mathcal{B})$ is at most $k$, then $\mathcal{B}$ is $k$-thin, but not necessarily vice-versa. Moreover, notice that every collection of internally disjoint balls in the plane is $2$-thin. 

Thus the following lemma is connected both to degree-boundedness and to the minimum degree of planar graphs. In particular, the lemma implies that for any fixed $d$, the class of intersection graphs of balls in $\mathbb{R}^d$ is degree-bounded by the linear function $f_d(\tau) = (2\tau+1)3^d$. This lemma is well-known, but since the proof is short, we provide it anyways.

\begin{lemma}
\label{lem:intBalls}
For any positive integers $k$ and $d$ and any finite $k$-thin collection $\mathcal{B}$ of balls in $\mathbb{R}^d$, the intersection graph $G(\mathcal{B})$ has a vertex of degree less than $k\cdot 3^d$.
\end{lemma}
\begin{proof}
Let $B \in \mathcal{B}$ be a ball of minimum volume; we will show that $B$ is the desired vertex of small degree. By translating and then scaling, we may assume that $B$ is the unit ball centered at the origin, that is, $B = \{x \in \mathbb{R}^d: \lVert x \rVert_2 \leq 1\}$. (A \emph{unit ball} is any ball of radius~$1$.) Now note that for any $B' \in \mathcal{B} \setminus\{B\}$ which intersects $B$, there exists a unit ball which is contained in $B'$ and still intersects $B$. Thus, there exists a $k$-thin collection $\mathcal{B}'$ of unit balls so that 1) the degree of $B$ in the intersection graph $G(\mathcal{B})$ is less than $|\mathcal{B}'|$, and 2) every ball in $\mathcal{B}'$ is a subset of the ball of radius~$3$ centered at the origin. 

Let us write $v_d$ for the volume of a unit ball in $\mathbb{R}^d$. It follows from the definition of $\mathcal{B}'$ that the sum of the volumes of the balls in $\mathcal{B}'$ is $v_d|\mathcal{B}'|$, and that this sum is at most $k(v_d \cdot 3^d)$ since $v_d \cdot 3^d$ is the volume of a ball of radius~$3$. So $|\mathcal{B}'| \leq k \cdot 3^d$, and the lemma follows.
\end{proof}

It is an interesting but likely very difficult question to find the optimal constant which can replace $3^d$ in Lemma~\ref{lem:intBalls}. While there are many ways to formulate this problem, let us pose the following version.

\begin{problem}
\label{prob:spherePacking}
For each fixed integer $d$, determine the smallest integer $c_d$ so that every finite $k$-thin collection $\mathcal{B}$ of balls in $\mathbb{R}^d$ contains a ball which intersects at most $c_dk+O_d(1)$ balls in~$\mathcal{B}$.
\end{problem}
\noindent While we do not know of any direct reduction, this problem seems highly related to the famous question about the density of sphere packings; see~\cite{spherePacking}. We suspect that $c_d$ must be exponential in $d$; perhaps this can be shown by including points of a very fine grid independently at random with an appropriately chosen probability. 

We also note that a similar argument to the one in Lemma~\ref{lem:intBalls} works for intersection graphs of compact convex sets in $\mathbb{R}^d$ of bounded aspect ratio. (The \emph{aspect ratio} of a compact convex set is its diameter divided by its height; the \emph{diameter} is the largest distance between two parallel hyperplanes which ``surround'' the set, and the \emph{height} is the smallest such distance.) There are also more sophisticated arguments that such intersection graphs are degree-bounded due to the existence of small separators~\cite{HarPeledQuanrud, MTTV, SmithWormald}. Dvo\v{r}\'{a}k, Norin, and the second author~\cite{DMN2021} also found a more general geometric condition that guarantees the existence of small separators and thus a vertex of small degree. The proof works by sorting the shapes by their volume, so it could also be used to show that any shape of minimum volume is a low-degree vertex.

\subsection{Circle graphs and vertex-minors}
A \emph{chord} of a circle is a (closed) line segment with ends on the circle. Let us consider intersection graphs of a (finite) collection of chords on the unit circle; any such graph is called a \emph{circle graph}. More precisely, a graph $G$ is a \emph{circle graph} if there exists a collection $\mathcal{C}$ of chords on the unit circle $\{(x,y) \in \mathbb{R}^2: x^2+y^2 = 1\}$ so that $G$ is isomorphic to the intersection graph $G(\mathcal{C})$.

Fox and Pach~\cite{foxPach2010} proved that the class of all circle graphs is linearly degree-bounded, that is, that the following bound holds.

\begin{theorem}[\cite{foxPach2010}]
There exists a constant $c$ so that every circle graph $G$ has a vertex of degree at most $c \cdot \tau(G)$.
\end{theorem}
\noindent In fact, Fox and Pach proved this theorem in much more generality; they showed that it holds 1) for intersection graphs of line segments in the plane, and 2) for intersection graphs of collections of strings in the plane so that every pair of strings intersects at most $t$ times. (The constant $c$ depends on $t$.) A \emph{string} is a homeomorphic image of the interval $[0,1]$ in the plane. We will discuss strings more in the next subsection.

The reason why we are discussing circle graphs in particular is related to a more general question about classes of graphs with a forbidden vertex-minor. A \emph{vertex-minor} of a graph $G$ is any graph that can be obtained from $G$ by taking induced subgraphs and by performing certain local moves at vertices called local complementations. \emph{Locally complementing} at a vertex $v$ of a graph $G$ replaces the induced subgraph on the neighborhood of $v$ by its complement. That is, it exchanges edges/non-edges between any pair of vertices $x,y \neq v$ which are adjacent to~$v$.

The class of circle graphs is closed under taking vertex-minors, and Bouchet~\cite{BouchetCircleChar} found the forbidden vertex-minors: there are only three of them. Surprisingly, this theorem is related to the famous theorem of Kuratowski and Wagner about the forbidden minors for planar graphs. The connection was discovered by de Fraysseix~\cite{planarAndCircle}, and Geelen and Oum~\cite{pivotMinorsCircle} have proven a common generalization of the two theorems. Motivated by a structural question of Geelen about classes of graphs which forbid a vertex-minor (see the second author's thesis for a formal statement of the conjecture~\cite{McCartyThesis}), we conjecture the following.

\begin{conjecture}
\label{conj:VM}
For any graph $H$, there exists a constant $c = c(H)$ so that every graph $G$ with no vertex-minor isomorphic to $H$ contains a vertex of degree at most $c \cdot \tau(G)$.
\end{conjecture}

The conjecture says that for any graph $H$, the class of all graphs with no vertex-minor isomorphic to $H$ is linearly degree-bounded. Such classes are known to be degree-bounded by a theorem of K\"{u}hn and Osthus~\cite{KO04induced} (which is stated as Theorem~\ref{thm:inducedSubdivision} in this paper). So Conjecture~\ref{conj:VM} is really about obtaining an optimal bound; even just a bound where the degree of the polynomial does not depend on $H$ would be interesting.  We note that Davies proved that classes with a forbidden vertex-minor are $\chi$-bounded~\cite{vmChiBound}, and that Conjecture~\ref{conj:VM} holds when $H$ is a circle graph by combining results in~\cite{gridThmVM, GW00, rankWidthTiedCW}.

\subsection{String graphs and induced minors}
\label{subsec:string}

A \emph{string} is a continuous arc in the plane, and a \emph{string graph} is any graph which is isomorphic to the intersection graph $G(\mathcal{S})$ of some finite collection $\mathcal{S}$ of strings. Such graphs have a long history, and we refer the reader to the survey of Matou\v{s}ek~\cite{stringSurvey} for more background.

For our purposes, string graphs are interesting because classes of string graphs which forbid a biclique as a subgraph behave a lot like planar graphs. In particular, they satisfy a separator theorem which generalizes the original separator theorem for planar graphs of Lipton and Tarjan~\cite{LiptonTarjan}. This theorem was conjectured by Fox and Pach~\cite{foxPach2010}, and Matou\v{s}ek~\cite{MatousekSeparators} proved a slightly weaker version with an extra $\log(m)$ factor. (The theorem of Matou\v{s}ek would have already been strong enough for our applications.) The stronger version stated here was proven by Lee~\cite{LeeSeparators}. A \emph{balanced separator} in a graph $G$ is a subset $X \subseteq V(G)$ so that once $X$ is deleted from $G$, every remaining component has at most $2|V(G)|/3$ vertices. (The fraction $2/3$ is not too important and can often be replaced by any fixed positive number less than $1$.)

\begin{theorem}[\cite{LeeSeparators}]
\label{thm:separators}
There exists a constant $c$ so that every string graph with $m$ edges has a balanced separator of size at most $c\sqrt{m}$.
\end{theorem}

Theorem~\ref{thm:separators} implies that the class of string graphs is degree-bounded. This implication can be obtained by using the K\H{o}v\'{a}ri-S\'{o}s-Tur\'{a}n Theorem~\cite{KST} and the fact that every graph class with (strongly) sublinear separators actually has polynomial expansion, which was proven by Dvo\v{r}\'{a}k and Norin~\cite{subLinearSepPolyExp}. However, a much more direct proof is given by Fox and Pach in~\cite[Theorem~5]{stringGraphsChi}. As noted by Lee~\cite{LeeSeparators}, the proof of Fox and Pach implies the following bound.

\begin{theorem}[\cite{LeeSeparators}]
\label{thm:stringDegree}
There exists a constant $c$ so that every string graph $G$ has a vertex of degree at most $c\cdot \tau(G)\log(\tau(G))$.
\end{theorem}

The separator theorem of Lee~\cite{LeeSeparators} applies to much more general classes of graphs. First of all, a \emph{minor} of a graph $G$ is any graph that can be obtained from $G$ by a sequence of vertex-deletions, edge-deletions, and edge-contractions. An \emph{induced minor} of $G$ is any graph that can be obtained from $G$ by just a sequence of vertex-deletions and edge-contractions (so edge-deletions are not allowed). Graph minors are very well-understood; the famous Graph Minors Theorem of Robertson and Seymour~\cite{graphMinors20} says that any class of graphs which is closed under taking minors has only finitely many excluded minors. Induced minors have been studied much less, but there has been a resurgence recently due to connections with metric geometry; see~\cite{conjCourseMenger, asymptoticDim, inducedMinorsConj, ceCourseMenger}. Notably, the class of string graphs is closed under taking induced minors, and yet Kratochv\'{\i}l~\cite{stringGraphsNPHard} proved that it is NP-hard to determine if a graph is a string graph. This theorem shows that induced minors are much harder to deal with than graph minors. Every minor-closed class can be recognized in polynomial time due to the Graph Minors Theorem and algorithms for testing for a fixed graph $H$ as a minor~\cite{algorithmMinors}.

We note that Lee~\cite{LeeSeparators} actually proved a separator theorem like Theorem~\ref{thm:separators} for any class of intersection graphs which come from a proper minor-closed class. (A class of graphs $\mathcal{F}$ is \emph{proper} if there exists a graph which is not isomorphic to any graph in  $\mathcal{F}$.) Given a class of graphs $\mathcal{F}$, we write $int(\mathcal{F})$ for the class of all graphs $G$ so that there exists a graph $\widehat{G} \in \mathcal{F}$ and a collection $\mathcal{C}$ of subsets of $V(\widehat{G})$ so that each set in $\mathcal{C}$ induces a connected subgraph of $\widehat{G}$, and so that $G$ is isomorphic to the intersection graph $G(\mathcal{C})$. It turns out that the class of all string graphs is precisely $int(\mathcal{F}_{planar})$, where $\mathcal{F}_{planar}$ denotes the class of all planar graphs; see~\cite{LeeSeparators}. In general, if $\mathcal{F}$ is any proper minor-closed class, then $int(\mathcal{F})$ is a proper induced-minor-closed class. Lee~\cite{LeeSeparators} proved that any class of the form $int(\mathcal{F})$, where $\mathcal{F}$ is a proper minor-closed class, satisfies a separator theorem like Theorem~\ref{thm:separators}.

This theorem of Lee implies that any such class $int(\mathcal{F})$ is degree-bounded. However, we actually already knew that by Theorem~\ref{thm:inducedSubdivision} of K\"{u}hn and Osthus~\cite{KO04induced} (because for any graph $H$, every induced subdivision of $H$ contains $H$ as an induced minor). However, we conjecture that something even stronger holds for classes of the form $int(\mathcal{F})$; we believe that there should be a uniform bound on the rate of growth of their degree-bounding functions. Perhaps this bound is even the same as in Theorem~\ref{thm:stringDegree}.

\begin{conjecture}
\label{conj:intMinorFree}
For any proper minor-closed class of graphs $\mathcal{F}$, there exists a constant $c = c(\mathcal{F})$ so that every graph $G \in int(\mathcal{F})$ has a vertex of degree at most $c\cdot \tau(G)\log(\tau(G))$.
\end{conjecture}
\noindent In fact, it seems that the only missing piece is a generalization of~\cite[Lemma~1.7]{stringGraphsChi}, which was originally proven in another paper by Fox and Pach~\cite{incompGraphs}.

Finally, we note that Korhonen and Lokshtanov~\cite{inducedMinorsSep} recently generalized the separator theorem even further, to any proper induced-minor-closed class. However, we believe that Conjecture~\ref{conj:intMinorFree} does not hold for arbitrary proper induced-minor-closed classes. This is due to the fact that classes of graphs with bounded independence number are actually closed under taking induced minors. This example was pointed out by Korhonen and Lokshtanov as one that does not seem to be contained in $int(\mathcal{F})$ for any proper minor-closed class~$\mathcal{F}$.

\subsection{Width parameters and a broader trend}
\label{subsec:beyond}

As pointed out in Theorem~\ref{thm:examples}, every graph class of bounded flip-width is degree-bounded due to a theorem of Toru\'{n}czyk~\cite{flipWidth}. This theorem implies that every class of bounded clique-width or bounded twin-width is degree-bounded, simply because such classes have bounded flip-width as well. However, the connections to degree-boundedness were proven even earlier for clique-width~\cite{GW00} and twin-width~\cite{twinWidthII}. It is also not hard to show that many other classes of bounded width are degree-bounded, including classes of bounded tree-independence number~\cite{treeIndNum}, mim-width~\cite{mimWidthI, mimWidthThesis}, and induced matching treewidth~\cite{indMatching24, indMatchingTW, yolov}.

These connections to width parameters actually all suggest a more general trend. Let us say that a class of graphs $\F$ is \textit{weakly sparse} if there exists an integer $t$ such that no graph in $\F$ has a subgraph isomorphic to $K_{t,t}$. This definition was introduced by Ne\v{s}et\v{r}il, {Ossona de Mendez}, Rabinovich, and Siebertz~\cite{NORS19} to help capture a common trend in structural graph theory. Namely, that properties of graph classes often come in pairs, with one property $P_{sparse}$ that is suitable in the sparse setting, and one property $P_{dense}$ which, informally, extends $P_{sparse}$ to the dense setting. (We think of a graph as being sparse if it has few edges and dense if it has many edges, informally.) These properties should agree on every weakly sparse graph class, that is, a class of graphs $\F$ should have property $P_{sparse}$ if and only if $\F$ has property $P_{dense}$ and $\F$ is weakly sparse. 

The following theorem combines some well-known examples of this trend. The first three bullet points are equivalently stated as~\cite[Theorem~2.10]{NORS19}, the fourth is from~\cite{twinWidthII}, the fifth is from~\cite{structBoundedDeg}, the sixth is stated as~\cite[Theorem~6.3]{flipWidth}, and the seventh is stated as~\cite[Corollary~2.3]{NORS19}. The final bullet point follows just from the definitions. We note that the third bullet point originates with Gurski and Wanke~\cite{GW00}. Shrubdepth was originally introduced as an analog of treedepth for dense graphs~\cite{shrubDepthIntroduced}, although the first bullet point was proven later in~\cite[Lemma~2.12]{structurallyBdExp}. Finally, we also note that the fifth and sixth bullet points are proven using a theorem of Dvo\v{r}\'{a}k~\cite{D18} as the main ingredient.

The first three bullet points all correspond to width parameters which are closely related to each other. While it may seem that there are an unbounded number of width parameters that one could consider, there is a real theory of width parameters which shows that (many, but not all) well-studied width parameters are instantiations of one abstract idea. This formalism is discussed in~\cite{branchDepthEtc}, for instance. We also refer the reader to a paper of Gajarsk\'{y}, Pilipczuk, and Toru\'{n}czyk~\cite{stableTwinWidth} for an explanation of how the following theorem is connected to transductions in the first-order logic of graphs. 

\begin{theorem}[\cite{twinWidthII, structBoundedDeg, structurallyBdExp, GW00, NORS19, flipWidth}]
The following statements hold for any graph class $\F$.\begin{itemize}
\item $\mathcal{F}$ has bounded treedepth if and only if $\mathcal{F}$ has bounded shrubdepth and is weakly sparse.
\item $\mathcal{F}$ has bounded pathwidth if and only if $\mathcal{F}$ has bounded linear rank-width and is weakly sparse.
\item $\mathcal{F}$ has bounded treewidth if and only if $\mathcal{F}$ has bounded rank-width and is weakly sparse.
\item $\mathcal{F}$ has bounded sparse twin-width if and only if $\mathcal{F}$ has bounded twin-width and is weakly sparse. 
\item $\mathcal{F}$ has bounded maximum degree if and only if $\mathcal{F}$ has structurally bounded maximum degree and is weakly sparse.
\item $\mathcal{F}$ has bounded expansion if and only if $\mathcal{F}$ has bounded flip-width and is weakly sparse.
\item $\mathcal{F}$ is nowhere dense if and only if $\mathcal{F}$ is monadically dependent and is weakly sparse.
\item $\mathcal{F}$ has bounded minimum degree if and only if $\mathcal{F}$ is degree-bounded and is weakly sparse.
\end{itemize}
\end{theorem}

The point here is that, in well-structured graph classes, forbidding a biclique as a subgraph is often enough to force a lot of structure (much more than just bounding the minimum degree). For instance, forbidding all induced subdivisions of a fixed graph $H$ not only makes the class degree-bounded as stated in Theorem~\ref{thm:inducedSubdivision}, but it also makes every weakly sparse subclass have bounded expansion; this is the theorem of Dvo\v{r}\'{a}k~\cite{D18}. If $H$ is a path, then the weakly sparse subclasses even have bounded treedepth~\cite{polyBoundPath}. Similarly, Wei{\ss}auer~\cite{W19tw} showed that if $H$ is a cycle, then the weakly sparse subclasses have bounded treewidth. Finally, the theorems about separators in Section~\ref{subsec:string} are yet another example of this trend.


\section*{Acknowledgments}
We are grateful to Bryce Frederickson for pointing out the theorem of Golumbic and Goss~\cite[Theorem~4]{bipChordal} in connection with degree-perfect graphs, to Tung Nguyen and Ant\'{o}nio Gir{\~{a}}o for critical corrections to Sections~\ref{sec:strongerExtremal} and~\ref{subsec:degNotChi}, to Paul Seymour and Raphael Steiner for helpful improvements to Section~\ref{subsec:geometric}, and to Shakhar Smorodinsky for discussions about the Zarankiewicz problem. We would also like to thank the referees for their careful comments that have improved the paper throughout.

\bibliographystyle{amsplain}
\bibliography{degree-bounded}

\providecommand{\bysame}{\leavevmode\hbox to3em{\hrulefill}\thinspace}
\providecommand{\MR}{\relax\ifhmode\unskip\space\fi MR }
\providecommand{\MRhref}[2]{%
  \href{http://www.ams.org/mathscinet-getitem?mr=#1}{#2}
}
\providecommand{\href}[2]{#2}
\begin{thebibliography}{100}

\bibitem{indMatching24}
T.~Abrishami, M.~Bria\'{n}ski, J.~Czy\.{z}ewska, R.~McCarty, M.~Milani\v{c},
  P.~Rz{\polhk{a}}{\.{z}}ewski, and B.~Walczak, \emph{Excluding a clique or a
  biclique in graphs of bounded induced matching treewidth},  (2024),
  arXiv:2405.04617.

\bibitem{AdlerAdler}
H.~Adler and I.~Adler, \emph{Interpreting nowhere dense graph classes as a
  classical notion of model theory}, European J. Combin. \textbf{36} (2014),
  322--330. \MR{3131898}

\bibitem{AKSRamsey}
M.~Ajtai, J.~Koml\'{o}s, and E.~Szemer\'{e}di, \emph{A note on {R}amsey
  numbers}, J. Combin. Theory Ser. A \textbf{29} (1980), no.~3, 354--360.
  \MR{600598}

\bibitem{conjCourseMenger}
S.~Albrechtsen, T.~Huynh, R.~W. Jacobs, P.~Knappe, and P.~Wollan, \emph{A
  {M}enger-type theorem for two induced paths}, SIAM J. Discrete Math.
  \textbf{38} (2024), no.~2, 1438--1450.

\bibitem{epsilonTNets}
N.~Alon, B.~Jartoux, C.~Keller, S.~Smorodinsky, and Y.~Yuditsky, \emph{The
  {$\varepsilon$}-{$t$}-net problem}, Discrete Comput. Geom. \textbf{68}
  (2022), no.~2, 618--644. \MR{4472580}

\bibitem{Andreev}
E.~M. Andreev, \emph{Convex polyhedra of finite volume in
  {L}oba\v{c}evski\u{\i} space}, Mat. Sb. (N.S.) \textbf{83(125)} (1970),
  256--260. \MR{273510}

\bibitem{orderedGraphsChi}
M.~Axenovich, J.~Rollin, and T.~Ueckerdt, \emph{Chromatic number of ordered
  graphs with forbidden ordered subgraphs}, Combinatorica \textbf{38} (2018),
  no.~5, 1021--1043. \MR{3884776}

\bibitem{semilinear2021}
A.~Basit, A.~Chernikov, S.~Starchenko, T.~Tao, and C.-M. Tran,
  \emph{Zarankiewicz's problem for semilinear hypergraphs}, Forum Math. Sigma
  \textbf{9} (2021), Paper No. e59, 23. \MR{4308822}

\bibitem{BFK2022}
I.~Benjamini, M.~Fraczyk, and G.~Kun, \emph{Expander spanning subgraphs with
  large girth}, Israel J. Math. \textbf{251} (2022), no.~1, 155--171.
  \MR{4555893}

\bibitem{bergePerfect}
C.~Berge, \emph{F\"{a}rbung von graphen, deren s\"{a}mtliche bzw. deren
  ungerade kreise starr sind}, Wiss. Z. Martin-Luther-Univ. Halle-Wittenberg
  Math.-Natur. Reihe \textbf{10} (1961), 114.

\bibitem{chiRamseyLower1}
T.~Bohman and P.~Keevash, \emph{Dynamic concentration of the triangle-free
  process}, Random Structures Algorithms \textbf{58} (2021), no.~2, 221--293.
  \MR{4201797}

\bibitem{asymptoticDim}
M.~Bonamy, N.~Bousquet, L.~Esperet, C.~Groenland, C.-H. Liu, F.~Pirot, and
  A.~Scott, \emph{Asymptotic dimension of minor-closed families and
  {A}ssouad–{N}agata dimension of surfaces}, J. Eur. Math. Soc. (2023),
  published online first.

\bibitem{polyBoundPath}
M.~Bonamy, N.~Bousquet, M.~Pilipczuk, P.~Rz{\polhk{a}}{\.{z}}ewski,
  S.~Thomass\'{e}, and B.~Walczak, \emph{Degeneracy of {$P_t$}-free and
  {$C_{\geqslant t}$}-free graphs with no large complete bipartite subgraphs},
  J. Combin. Theory Ser. B \textbf{152} (2022), 353--378. \MR{4332745}

\bibitem{rankDecomp2}
M.~Bonamy and M.~Pilipczuk, \emph{Graphs of bounded cliquewidth are
  polynomially {$\chi$}-bounded}, Advances in Combinatorics (2020).

\bibitem{delineationTwinWidth}
\'{E}. Bonnet, D.~Chakraborty, E.~J. Kim, N.~K\"{o}hler, R.~Lopes, and
  S.~Thomass\'{e}, \emph{Twin-width {VIII}: delineation and win-wins}, 17th
  {I}nternational {S}ymposium on {P}arameterized and {E}xact {C}omputation,
  LIPIcs. Leibniz Int. Proc. Inform., vol. 249, Schloss Dagstuhl. Leibniz-Zent.
  Inform., Wadern, 2022, pp.~Paper No. 9, 18. \MR{4526730}

\bibitem{twinWidthIII}
\'{E}. Bonnet, C.~Geniet, E.~J. Kim, S.~Thomass\'{e}, and R.~Watrigant,
  \emph{Twin-width {III}: max independent set, min dominating set, and
  coloring}, 48th {I}nternational {C}olloquium on {A}utomata, {L}anguages, and
  {P}rogramming, LIPIcs. Leibniz Int. Proc. Inform., vol. 198, Schloss
  Dagstuhl. Leibniz-Zent. Inform., Wadern, 2021, pp.~Art. No. 35, 20.
  \MR{4288865}

\bibitem{twinWidthII}
\'{E}. Bonnet{}, C.~Geniet, E.~J. Kim, S.~Thomass\'{e}, and R.~Watrigant,
  \emph{Twin-width {II}: small classes}, Comb. Theory \textbf{2} (2022), no.~2,
  Paper No. 10, 42. \MR{4449818}

\bibitem{BouchetCircleChar}
A.~Bouchet, \emph{Circle graph obstructions}, J. Combin. Theory Ser. B
  \textbf{60} (1994), no.~1, 107--144. \MR{1256586}

\bibitem{polyBoundednessBip}
R.~Bourneuf, M.~Buci\'{c}, L.~Cook, and J.~Davies, \emph{On polynomial
  degree-boundedness},  (2023), arXiv:2311.03341v2.

\bibitem{twinWidthColoring3}
R.~Bourneuf and S.~Thomass\'{e}, \emph{Bounded twin-width graphs are
  polynomially {$\chi$}-bounded},  (2023), arXiv:2303.11231.

\bibitem{notPolyChi}
M.~Bria\'{n}ski, J.~Davies, and B.~Walczak, \emph{Separating polynomial
  $\chi$-boundedness from $\chi$-boundedness}, Combinatorica (2023),
  https://doi.org/10.1007/s00493--023--00054--3.

\bibitem{niceCirclePacking}
G.~R. Brightwell and E.~R. Scheinerman, \emph{Representations of planar
  graphs}, SIAM J. Discrete Math. \textbf{6} (1993), no.~2, 214--229.
  \MR{1215229}

\bibitem{loglogEH}
M.~Buci\'{c}, T.~Nguyen, A.~Scott, and P.~Seymour, \emph{Induced subgraph
  density. {I}. a loglog step towards {E}rd{\H{o}}s-{H}ajnal},  (2023),
  arXiv:2301.10147.

\bibitem{burlingGraphs}
J.~Burling, \emph{On coloring problems of families of polytopes}, PhD thesis,
  University of Colorado (1965).

\bibitem{inspoNotPolyChi}
A.~Carbonero, P.~Hompe, B.~Moore, and S.~Spirkl, \emph{A counterexample to a
  conjecture about triangle-free induced subgraphs of graphs with large
  chromatic number}, J. Combin. Theory Ser. B \textbf{158} (2023), 63--69.
  \MR{4484828}

\bibitem{restrictedFrameGraphs}
J.~Chalopin, L.~Esperet, Z.~Li, and P.~Ossona~de Mendez, \emph{Restricted frame
  graphs and a conjecture of {S}cott}, Electron. J. Combin. \textbf{23} (2016),
  no.~1, Paper 1.30, 21. \MR{3484735}

\bibitem{EHsurvey}
M.~Chudnovsky, \emph{The {E}rd{\H{o}}s–{H}ajnal {C}onjecture—{A} survey},
  J. Graph Theory \textbf{75} (2014), no.~2, 178--190.

\bibitem{strongPerfect}
M.~Chudnovsky, N.~Robertson, P.~Seymour, and R.~Thomas, \emph{The strong
  perfect graph theorem}, Ann. of Math. (2) \textbf{164} (2006), no.~1,
  51--229. \MR{2233847}

\bibitem{longHolesChiBd}
M.~Chudnovsky, A.~Scott, and P.~Seymour, \emph{Induced subgraphs of graphs with
  large chromatic number. {III}. {L}ong holes}, Combinatorica \textbf{37}
  (2017), no.~6, 1057--1072. \MR{3759907}

\bibitem{chandeliersStrings}
M.~Chudnovsky, A.~Scott, and P.~Seymour{}{}, \emph{Induced subgraphs of graphs
  with large chromatic number. {V}. {C}handeliers and strings}, J. Combin.
  Theory Ser. B \textbf{150} (2021), 195--243. \MR{4270097}

\bibitem{longOddHoles}
M.~Chudnovsky, A.~Scott, P.~Seymour, and S.~Spirkl, \emph{Induced subgraphs of
  graphs with large chromatic number. {VIII}. {L}ong odd holes}, J. Combin.
  Theory Ser. B \textbf{140} (2020), 84--97. \MR{4033097}

\bibitem{spherePacking}
H.~Cohn, \emph{A conceptual breakthrough in sphere packing}, Notices Amer.
  Math. Soc. \textbf{64} (2017), no.~2, 102--115. \MR{3587715}

\bibitem{perfectMatrices}
G.~Cornu\'{e}jols, \emph{Combinatorial optimization}, ch.~3. Perfect Graphs and
  Matrices, pp.~21--30, SIAM, CBMS-NSF Regional Conference Series in Applied
  Mathematics, 2001.

\bibitem{treeIndNum}
C.~Dallard, M.~Milani\v{c}, and K.~\v{S}torgel, \emph{Treewidth versus clique
  number. {II}. {T}ree-independence number}, J. Combin. Theory Ser. B
  \textbf{164} (2024), 404--442. \MR{4664382}

\bibitem{chiRamseyUpper}
E.~Davies and F.~Illingworth, \emph{The {$\chi$}-{R}amsey problem for
  triangle-free graphs}, SIAM J. Discrete Math. \textbf{36} (2022), no.~2,
  1124--1134. \MR{4413062}

\bibitem{circleGraphsOpt}
J.~Davies, \emph{Improved bounds for colouring circle graphs}, Proc. Amer.
  Math. Soc. \textbf{150} (2022), no.~12, 5121--5135. \MR{4494591}

\bibitem{vmChiBound}
J.~Davies{}, \emph{Vertex-minor-closed classes are {$\chi$}-bounded},
  Combinatorica \textbf{42} (2022), 1049--1079. \MR{4543723}

\bibitem{polygonVisibility}
J.~Davies{}, T.~Krawczyk, R.~McCarty, and B.~Walczak, \emph{Colouring polygon
  visibility graphs and their generalizations}, 37th {I}nternational
  {S}ymposium on {C}omputational {G}eometry, LIPIcs. Leibniz Int. Proc.
  Inform., vol. 189, Schloss Dagstuhl. Leibniz-Zent. Inform., Wadern, 2021,
  pp.~Art. No. 29, 16. \MR{4287036}

\bibitem{curveVisibility}
J.~Davies, T.~Krawczyk, R.~McCarty, and B.~Walczak, \emph{Colouring polygon
  visibility graphs and their generalizations}, 37th {I}nternational
  {S}ymposium on {C}omputational {G}eometry, LIPIcs. Leibniz Int. Proc.
  Inform., vol. 189, Schloss Dagstuhl. Leibniz-Zent. Inform., Wadern, 2021,
  pp.~Art. No. 29, 16. \MR{4287036}

\bibitem{groundedLgraphs2}
J.~Davies{}{}, T.~Krawczyk, R.~McCarty, and B.~Walczak, \emph{Grounded
  {L}-graphs are polynomially {$\chi$}-bounded}, Discrete Comput. Geom.
  \textbf{70} (2023), no.~4, 1523--1550. \MR{4670367}

\bibitem{circleGraphsQuad}
J.~Davies and R.~McCarty, \emph{Circle graphs are quadratically
  {$\chi$}-bounded}, Bull. Lond. Math. Soc. \textbf{53} (2021), no.~3,
  673--679. \MR{4275079}

\bibitem{planarAndCircle}
H.~de~Fraysseix, \emph{Local complementation and interlacement graphs},
  Discrete Math. \textbf{33} (1981), no.~1, 29--35. \MR{597225}

\bibitem{DKMR11}
D.~Dellamonica, Jr., V.~Koubek, D.~M. Martin, and V.~R\"{o}dl, \emph{On a
  conjecture of {T}homassen concerning subgraphs of large girth}, J. Graph
  Theory \textbf{67} (2011), no.~4, 316--331. \MR{2839370}

\bibitem{DR11}
D.~Dellamonica, Jr. and V.~R\"{o}dl, \emph{A note on {T}homassen's conjecture},
  J. Combin. Theory Ser. B \textbf{101} (2011), no.~6, 509--515. \MR{2832817}

\bibitem{branchDepthEtc}
M.~DeVos, O.~Kwon, and S.~Oum, \emph{Branch-depth: generalizing tree-depth of
  graphs}, European J. Combin. \textbf{90} (2020), 103186, 23. \MR{4129026}

\bibitem{singlyExponential}
X.~Du, A.~Gir{\~{a}}o, Z.~Hunter, R.~McCarty, and A.~Scott, \emph{Induced
  {$C_4$}-free subgraphs with large average degree},  (2023),
  arXiv:2307.08361v2.

\bibitem{D18}
Z.~Dvo\v{r}\'{a}k, \emph{Induced subdivisions and bounded expansion}, European
  J. Combin. \textbf{69} (2018), 143--148. \MR{3738146}

\bibitem{rankDecomp1}
Z.~Dvo\v{r}\'{a}k and D.~Kr\'{a}l’, \emph{Classes of graphs with small rank
  decompositions are {$\chi$}-bounded}, European Journal of Combinatorics
  \textbf{33} (2012), no.~4, 679--683.

\bibitem{DMN2021}
Z.~Dvo\v{r}\'{a}k, R.~McCarty, and S.~Norin, \emph{Sublinear separators in
  intersection graphs of convex shapes}, SIAM J. Discrete Math. \textbf{35}
  (2021), no.~2, 1149--1164. \MR{4267483}

\bibitem{subLinearSepPolyExp}
Z.~Dvo\v{r}\'{a}k and S.~Norin, \emph{Strongly sublinear separators and
  polynomial expansion}, SIAM J. Discrete Math. \textbf{30} (2016), no.~2,
  1095--1101. \MR{3504982}

\bibitem{ErdosLargeChi}
P.~Erd\H{o}s, \emph{Graph theory and probability}, Canadian J. Math.
  \textbf{11} (1959), 34--38. \MR{102081}

\bibitem{EHConjecture}
P.~Erd\H{o}s and A.~Hajnal, \emph{Ramsey-type theorems}, vol.~25, 1989,
  Combinatorics and complexity (Chicago, IL, 1987), pp.~37--52. \MR{1031262}

\bibitem{ErdosSimonovits}
P.~Erd\H{o}s and M.~Simonovits, \emph{Some extremal problems in graph theory},
  Combinatorial theory and its applications, {I}-{III} ({P}roc. {C}olloq.,
  {B}alatonf\"ured, 1969), Colloq. Math. Soc. J\'anos Bolyai, vol.~4,
  North-Holland, Amsterdam-London, 1970, pp.~377--390.

\bibitem{esperet2017habilitation}
L.~Esperet, \emph{Graph colorings, flows and perfect matchings}, Habilitation
  thesis, Universit\'{e} Grenoble Alpe (2017), 24.

\bibitem{EKT19}
L.~Esperet, R.~J. Kang, and S.~Thomass\'{e}, \emph{Separation choosability and
  dense bipartite induced subgraphs}, Combin. Probab. Comput. \textbf{28}
  (2019), no.~5, 720--732. \MR{3991384}

\bibitem{chiRamseyLower2}
G.~Fiz~Pontiveros, S.~Griffiths, and R.~Morris, \emph{The triangle-free process
  and the {R}amsey number {$R(3,k)$}}, Mem. Amer. Math. Soc. \textbf{263}
  (2020), no.~1274, v+125. \MR{4073152}

\bibitem{foxPach2010}
J.~Fox and J.~Pach, \emph{A separator theorem for string graphs and its
  applications}, Combinatorics, Probability and Computing \textbf{19} (2010),
  no.~3, 371–390.

\bibitem{incompGraphs}
J.~Fox{} and J.~Pach, \emph{String graphs and incomparability graphs}, Adv.
  Math. \textbf{230} (2012), no.~3, 1381--1401. \MR{2921183}

\bibitem{stringGraphsChi}
J.~Fox and J.~Pach{}, \emph{Applications of a new separator theorem for string
  graphs}, Combin. Probab. Comput. \textbf{23} (2014), no.~1, 66--74.
  \MR{3197967}

\bibitem{FPSSZ2017}
J.~Fox, J.~Pach, A.~Sheffer, A.~Suk, and J.~Zahl, \emph{A semi-algebraic
  version of {Z}arankiewicz's problem}, J. Eur. Math. Soc. (JEMS) \textbf{19}
  (2017), no.~6, 1785--1810. \MR{3646875}

\bibitem{betterChiStringTriangleFree}
J.~Fox, J.~Pach, and A.~Suk, \emph{Quasiplanar graphs, string graphs, and the
  {E}rd{\H{o}}s-{G}allai problem}, Graph drawing and network visualization,
  Lecture Notes in Comput. Sci., vol. 13764, Springer, Cham, [2023] \copyright
  2023, pp.~219--231. \MR{4540414}

\bibitem{zarankiewiczsemialgebraic}
N.~Frankl and A.~Kupavskii, \emph{On the {E}rd{\H{o}}s-{P}urdy problem and the
  {Z}arankiewitz problem for semialgebraic graphs},  (2021), arXiv:2112.10245.

\bibitem{Furedi83}
Z.~F\"{u}redi, \emph{On finite set-systems whose every intersection is a kernel
  of a star}, Discrete Math. \textbf{47} (1983), no.~1, 129--132. \MR{720617}

\bibitem{structBoundedDeg}
J.~Gajarsk\'{y}, P.~Hlin\v{e}n\'{y}, J.~Obdr\v{z}\'{a}lek, D.~Lokshtanov, and
  M.~S. Ramanujan, \emph{A new perspective on {FO} model checking of dense
  graph classes}, ACM Trans. Comput. Log. \textbf{21} (2020), no.~4, Art. 28,
  23. \MR{4170050}

\bibitem{structurallyBdExp}
J.~Gajarsk\'{y}, S.~Kreutzer, J.~Ne\v{s}ET\v{r}il, P.~Ossona~de Mendez,
  M.~Pilipczuk, S.~Siebertz, and S.~Toru\'{n}czyk, \emph{First-order
  interpretations of bounded expansion classes}, ACM Trans. Comput. Logic
  \textbf{21} (2020), no.~4.

\bibitem{stableTwinWidth}
J.~Gajarsk\'{y}, M.~Pilipczuk, and S.~Toru\'{n}czyk, \emph{Stable graphs of
  bounded twin-width}, Proceedings of the 37th {A}nnual {ACM}/{IEEE}
  {S}ymposium on {L}ogic in {C}omputer {S}cience, ACM, New York, [2022]
  \copyright 2022, pp.~[Article 39], 12. \MR{4537035}

\bibitem{shrubDepthIntroduced}
R.~Ganian, P.~Hlin{\v{e}}n{\'y}, J.~Ne{\v{s}}et{\v{r}}il,
  J.~Obdr{\v{z}}{\'a}lek, P.~Ossona~de Mendez, and R.~Ramadurai, \emph{When
  trees grow low: Shrubs and fast {MSO1}}, Mathematical Foundations of Computer
  Science 2012 (Berlin, Heidelberg) (B.~Rovan, V.~Sassone, and P.~Widmayer,
  eds.), Springer Berlin Heidelberg, 2012, pp.~419--430.

\bibitem{gridThmVM}
J.~Geelen, O.~Kwon, R.~McCarty, and P.~Wollan, \emph{The grid theorem for
  vertex-minors}, J. Combin. Theory Ser. B \textbf{158} (2023), 93--116.
  \MR{4513819}

\bibitem{pivotMinorsCircle}
J.~Geelen and S.~Oum, \emph{Circle graph obstructions under pivoting}, J. Graph
  Theory \textbf{61} (2009), no.~1, 1--11. \MR{2514095}

\bibitem{inducedMinorsConj}
A.~Georgakopoulos and P.~Papasoglu, \emph{Graph minors and metric spaces},
  (2023), arXiv:2305.07456.

\bibitem{isOfis}
A.~Gir\~{a}o, F.~Illingworth, E.~Powierski, M.~Savery, A.~Scott, Y.~Tamitegama,
  and J.~Tan, \emph{Induced {S}ubgraphs of {I}nduced {S}ubgraphs of {L}arge
  {C}hromatic {N}umber}, Combinatorica \textbf{44} (2024), no.~1, 37--62.
  \MR{4707564}

\bibitem{polyBoundedness}
A.~Gir{\~{a}}o and Z.~Hunter, \emph{Induced subdivisions in {$K_{s,s}$}-free
  graphs with polynomial average degree},  (2024), arXiv:2310.18452v3.

\bibitem{bipChordal}
M.~C. Golumbic and C.~F. Goss, \emph{Perfect elimination and chordal bipartite
  graphs}, J. Graph Theory \textbf{2} (1978), no.~2, 155--163. \MR{493395}

\bibitem{GW00}
F.~Gurski and E.~Wanke, \emph{The tree-width of clique-width bounded graphs
  without {$K_{n,n}$}}, Graph-theoretic concepts in computer science
  ({K}onstanz, 2000), Lecture Notes in Comput. Sci., vol. 1928, Springer,
  Berlin, 2000, pp.~196--205. \MR{1850348}

\bibitem{GSConjecture1}
A.~Gy\'{a}rf\'{a}s, \emph{On {R}amsey covering-numbers}, Infinite and finite
  sets ({C}olloq., {K}eszthely, 1973; dedicated to {P}. {E}rd\H{o}s on his 60th
  birthday), {V}ols. {I}, {II}, {III}, Colloq. Math. Soc. J\'{a}nos Bolyai,
  vol. Vol. 10, North-Holland, Amsterdam-London, 1975, pp.~801--816.
  \MR{382051}

\bibitem{gyarfas1985chromatic}
A.~Gy{\'a}rf{\'a}s, \emph{On the chromatic number of multiple interval graphs
  and overlap graphs}, Discrete Math. \textbf{55} (1985), no.~2, 161--166.

\bibitem{GyarfasConjecture}
A.~Gy\'{a}rf\'{a}s{}, \emph{Problems from the world surrounding perfect
  graphs}, Proceedings of the {I}nternational {C}onference on {C}ombinatorial
  {A}nalysis and its {A}pplications ({P}okrzywna, 1985), vol.~19, 1987,
  pp.~413--441. \MR{951359}

\bibitem{treesAndChi}
A.~Gy\'{a}rf\'{a}s, E.~Szemer\'{e}di, and Zs. Tuza, \emph{Induced subtrees in
  graphs of large chromatic number}, Discrete Math. \textbf{30} (1980), no.~3,
  235--244. \MR{573638}

\bibitem{HarPeledQuanrud}
S.~Har-Peled and K.~Quanrud, \emph{Approximation algorithms for
  polynomial-expansion and low-density graphs}, SIAM Journal on Computing
  \textbf{46} (2017), no.~6, 1712--1744.

\bibitem{KSTHereditaryConj}
Z.~Hunter, A.~Milojevi\'{c}, B.~Sudakov, and I.~Tomon,
  \emph{K{\H{o}}v\'ari-{S}\'os-{T}ur\'an {t}heorem for hereditary families},
  (2024), arXiv:2401.10853v1.

\bibitem{mimWidthI}
L.~Jaffke, O.~Kwon, and J.~A. Telle, \emph{Mim-{W}idth {I}. {I}nduced path
  problems}, Discrete Appl. Math. \textbf{278} (2020), 153--168. \MR{4087200}

\bibitem{zarankiewiczVCdim}
O.~Janzer and C.~Pohoata, \emph{On the {Z}arankiewicz problem for graphs with
  bounded {VC}-dimension}, Combinatorica \textbf{44} (2024), 839--848.

\bibitem{ErdosSauer}
O.~Janzer and B.~Sudakov, \emph{Resolution of the {E}rd{\H{o}}s-{S}auer problem
  on regular subgraphs}, Forum Math. Pi \textbf{11} (2023), Paper No. e19, 13.
  \MR{4620330}

\bibitem{KSTepsilonNets}
C.~Keller and S.~Smorodinsky, \emph{{Zarankiewicz’s Problem via
  {$\epsilon$}-$t$-Nets}}, 40th International Symposium on Computational
  Geometry (SoCG 2024) (Dagstuhl, Germany) (Wolfgang Mulzer and Jeff~M.
  Phillips, eds.), Leibniz International Proceedings in Informatics (LIPIcs),
  vol. 293, Schloss Dagstuhl -- Leibniz-Zentrum f{\"u}r Informatik, 2024,
  pp.~66:1--66:15.

\bibitem{KST}
T.~K\H{o}v\'{a}ri, V.~T. S\'{o}s, and P.~Tur\'{a}n, \emph{On a problem of {K}.
  {Z}arankiewicz}, Colloq. Math. \textbf{3} (1954), 50--57. \MR{65617}

\bibitem{kiersteadPenrice}
H.~A. Kierstead and S.~G. Penrice, \emph{Radius two trees specify
  {$\chi$}-bounded classes}, J. Graph Theory \textbf{18} (1994), no.~2,
  119--129. \MR{1258244}

\bibitem{Koebe}
P.~Koebe, \emph{Kontaktprobleme der konformen abbildung}, Ber. Verh. S\"{a}chs.
  Akad. Leipzig 88, 1936.

\bibitem{inducedMinorsSep}
T.~Korhonen and D.~Lokshtanov, \emph{Induced-minor-free graphs: Separator
  theorem, subexponential algorithms, and improved hardness of recognition},
  pp.~5249--5275, 2024.

\bibitem{kostochka1997covering}
A.~Kostochka and J.~Kratochv{\'\i}l, \emph{Covering and coloring polygon-circle
  graphs}, Discrete Math. \textbf{163} (1997), no.~1-3, 299--305.

\bibitem{stringGraphsNPHard}
J.~Kratochv\'{\i}l, \emph{String graphs. {II}. {R}ecognizing string graphs is
  {NP}-hard}, J. Combin. Theory Ser. B \textbf{52} (1991), no.~1, 67--78.
  \MR{1109420}

\bibitem{KO04}
D.~K\"{u}hn and D.~Osthus, \emph{Every graph of sufficiently large average
  degree contains a {$C_4$}-free subgraph of large average degree},
  Combinatorica \textbf{24} (2004), no.~1, 155--162. \MR{2057689}

\bibitem{KO04induced}
D.~K\"{u}hn and D.{} Osthus, \emph{Induced subdivisions in {$K_{s,s}$}-free
  graphs of large average degree}, Combinatorica \textbf{24} (2004), no.~2,
  287--304. \MR{2071336}

\bibitem{KLST20}
M.~Kwan, S.~Letzter, B.~Sudakov, and T.~Tran, \emph{Dense induced bipartite
  subgraphs in triangle-free graphs}, Combinatorica \textbf{40} (2020), no.~2,
  283--305. \MR{4085991}

\bibitem{LeeSeparators}
J.~R. Lee, \emph{Separators in region intersection graphs}, 8th {I}nnovations
  in {T}heoretical {C}omputer {S}cience {C}onference, LIPIcs. Leibniz Int.
  Proc. Inform., vol.~67, Schloss Dagstuhl. Leibniz-Zent. Inform., Wadern,
  2017, pp.~Art. No. 1, 8. \MR{3754925}

\bibitem{indMatchingTW}
P.~T. Lima, M.~Milani\v{c}, P.~Mur\v{s}i\v{c}, K.~Okrasa,
  P.~Rz{\polhk{a}}{\.{z}}ewski, and K.~\v{S}torgel, \emph{Tree decompositions
  meet induced matchings: beyond {M}ax {W}eight {I}ndependent {S}et},  (2024),
  arXiv:2402.15834.

\bibitem{LiptonTarjan}
R.~J. Lipton and R.~E. Tarjan, \emph{A separator theorem for planar graphs},
  SIAM J. Appl. Math. \textbf{36} (1979), no.~2, 177--189. \MR{524495}

\bibitem{BroomFree2}
X.~Liu, J.~Schroeder, Z.~Wang, and X.~Yu, \emph{Polynomial {$\chi$}-binding
  functions for {$t$}-broom-free graphs}, J. Combin. Theory Ser. B \textbf{162}
  (2023), 118--133. \MR{4589655}

\bibitem{LTTZ2018}
P.-S. Loh, M.~Tait, C.~Timmons, and R.~M. Zhou, \emph{Induced {T}ur\'{a}n
  numbers}, Combin. Probab. Comput. \textbf{27} (2018), no.~2, 274--288.
  \MR{3778203}

\bibitem{weakPerfect}
L.~Lov\'{a}sz, \emph{Normal hypergraphs and the weak perfect graph conjecture},
  Topics on perfect graphs, North-Holland Math. Stud., vol.~88, North-Holland,
  Amsterdam, 1984, pp.~29--42. \MR{778747}

\bibitem{MatousekSeparators}
J.~Matou\v{s}ek, \emph{Near-optimal separators in string graphs}, Combin.
  Probab. Comput. \textbf{23} (2014), no.~1, 135--139. \MR{3197972}

\bibitem{stringSurvey}
J.~Matou\v{s}ek{}, \emph{String graphs and separators}, Geometry, structure and
  randomness in combinatorics, CRM Series, vol.~18, Ed. Norm., Pisa, 2015,
  pp.~61--97. \MR{3362300}

\bibitem{inducedBip}
R.~McCarty, \emph{Dense induced subgraphs of dense bipartite graphs}, SIAM J.
  Discrete Math. \textbf{35} (2021), no.~2, 661--667. \MR{4241510}

\bibitem{McCartyThesis}
R.~McCarty{}, \emph{Local structure for vertex-minors}, {P}h{D} thesis,
  University of Waterloo (2021).

\bibitem{groundedLgraphs1}
S.~McGuinness, \emph{On bounding the chromatic number of {L}-graphs}, Discrete
  Math. \textbf{154} (1996), no.~1-3, 179--187. \MR{1395457}

\bibitem{stringsBoundedIntersectionsChi}
S.~McGuinness{}, \emph{Colouring arcwise connected sets in the plane. {I}},
  Graphs Combin. \textbf{16} (2000), no.~4, 429--439. \MR{1804342}

\bibitem{MTTV}
G.~L. Miller, S.H. Teng, W.~Thurston, and S.~A. Vavasis, \emph{Separators for
  sphere-packings and nearest neighbor graphs}, J. ACM \textbf{44} (1997),
  no.~1, 1--29. \MR{1438463}

\bibitem{MPS2021}
R.~Montgomery, A.~Pokrovskiy, and B.~Sudakov, \emph{{$C_4$}-free subgraphs with
  large average degree}, Israel J. Math. \textbf{246} (2021), no.~1, 55--71.
  \MR{4358273}

\bibitem{NORS19}
J.~Ne\v{s}et\v{r}il, P.~{Ossona de Mendez}, R.~Rabinovich, and S.~Siebertz,
  \emph{Classes of graphs with low complexity: The case of classes with bounded
  linear rankwidth}, European J. Combin. \textbf{91} (2021), 103223, Colorings
  and structural graph theory in context (a tribute to Xuding Zhu).

\bibitem{almostEHpaths}
T.~Nguyen, A.~Scott, and P.~Seymour, \emph{Induced subgraph density. {V}. {A}ll
  paths approach {E}rd{\H{o}}s-{H}ajnal},  (2023), arXiv:2307.15032.

\bibitem{EHP5}
T.~Nguyen{}, A.~Scott, and P.~Seymour, \emph{Induced subgraph density. {VII}.
  the five-vertex path},  (2023), arXiv:2312.15333.

\bibitem{ceCourseMenger}
T.~Nguyen, A.~Scott, and P.~Seymour, \emph{A counterexample to the coarse
  {M}enger conjecture},  (2024), arXiv:2401.06685.

\bibitem{EHVCdimension}
T.{}{} Nguyen, A.~Scott, and P.~Seymour, \emph{Induced subgraph density. {VI}.
  {B}ounded {VC}-dimension},  (2024), arXiv:2312.15572v2.

\bibitem{almostLinearSubdivision}
T.{}{}{} Nguyen, A.~Scott, and P.~Seymour, \emph{Subdivisions and near-linear
  stable sets}, preprint (2024), 2409.09400.

\bibitem{almostLinearTrees}
T.{}{}{}{} Nguyen, A.~Scott, and P.~Seymour, \emph{Trees and near-linear stable
  sets},  (2024), arXiv:2409.09397.

\bibitem{rankWidthTiedCW}
S.~Oum and P.~Seymour, \emph{Approximating clique-width and branch-width}, J.
  Combin. Theory Ser. B \textbf{96} (2006), no.~4, 514--528. \MR{2232389}

\bibitem{representBurling}
A.~Pawlik, J.~Kozik, T.~Krawczyk, M.~Laso\'{n}, P.~Micek, W.~T. Trotter, and
  B.~Walczak, \emph{Triangle-free intersection graphs of line segments with
  large chromatic number}, J. Combin. Theory Ser. B \textbf{105} (2014), 6--10.

\bibitem{twinWidthColoring2}
M.~Pilipczuk and M.~Soko{\l}owski, \emph{Graphs of bounded twin-width are
  quasi-polynomially {$\chi$}-bounded}, J. Combin. Theory Ser. B \textbf{161}
  (2023), 382--406. \MR{4568111}

\bibitem{PRS95}
L.~Pyber, V.~R\"{o}dl, and E.~Szemer\'{e}di, \emph{Dense graphs without
  {$3$}-regular subgraphs}, J. Combin. Theory Ser. B \textbf{63} (1995), no.~1,
  41--54. \MR{1309356}

\bibitem{algorithmMinors}
N.~Robertson and P.~D. Seymour, \emph{Graph minors. {XIII}. {T}he disjoint
  paths problem}, J. Combin. Theory Ser. B \textbf{63} (1995), no.~1, 65--110.
  \MR{1309358}

\bibitem{graphMinors20}
N.~Robertson{} and P.~D. Seymour, \emph{Graph minors. {XX}. {W}agner's
  conjecture}, J. Combin. Theory Ser. B \textbf{92} (2004), no.~2, 325--357.
  \MR{2099147}

\bibitem{R77}
V.~R\"{o}dl, \emph{On the chromatic number of subgraphs of a given graph},
  Proc. Amer. Math. Soc. \textbf{64} (1977), no.~2, 370--371. \MR{469806}

\bibitem{Sauer}
N.~Sauer, \emph{On the density of families of sets}, J. Combinatorial Theory
  Ser. A \textbf{13} (1972), 145--147. \MR{307902}

\bibitem{subdivisionTreesChi}
A.~Scott, \emph{Induced trees in graphs of large chromatic number}, J. Graph
  Theory \textbf{24} (1997), no.~4, 297--311.

\bibitem{bananaTrees}
A.~Scott and P.~Seymour, \emph{Induced subgraphs of graphs with large chromatic
  number. {VI}. {B}anana trees}, J. Combin. Theory Ser. B \textbf{145} (2020),
  487--510. \MR{4152772}

\bibitem{ss20survey}
A.~Scott and P.~Seymour{}{}{}, \emph{A survey of $\chi$-boundedness}, J. Graph
  Theory \textbf{95} (2020), no.~3, 473--504.

\bibitem{DoubleStarFree2}
A.~Scott{}{}, P.~Seymour, and S.~Spirkl, \emph{Polynomial bounds for chromatic
  number. {III}. {E}xcluding a double star}, J. Graph Theory \textbf{101}
  (2022), no.~2, 323--340. \MR{4472775}

\bibitem{polyBoundTree}
A.~Scott{}, P.~Seymour, and S.~Spirkl, \emph{Polynomial bounds for chromatic
  number. {I}. {E}xcluding a biclique and an induced tree}, J. Graph Theory
  \textbf{102} (2023), no.~3, 458--471. \MR{4563201}

\bibitem{chiP5}
A.~Scott{}{}{}, P.~Seymour, and S.~Spirkl, \emph{Polynomial bounds for
  chromatic number. {IV}: {A} near-polynomial bound for excluding the
  five-vertex path}, Combinatorica \textbf{43} (2023), no.~5, 845--852.
  \MR{4648583}

\bibitem{copsAndRobbersTW}
P.~D. Seymour and R.~Thomas, \emph{Graph searching and a min-max theorem for
  tree-width}, J. Combin. Theory Ser. B \textbf{58} (1993), no.~1, 22--33.
  \MR{1214888}

\bibitem{Shelah}
S.~Shelah, \emph{A combinatorial problem; stability and order for models and
  theories in infinitary languages}, Pacific J. Math. \textbf{41} (1972),
  247--261. \MR{307903}

\bibitem{SmithWormald}
W.D. Smith and N.C. Wormald, \emph{Geometric separator theorems and
  applications}, Proceedings 39th Annual Symposium on Foundations of Computer
  Science (Cat. No.98CB36280), 1998, pp.~232--243.

\bibitem{zarankiewiczSurvey}
S.~Smorodinsky, \emph{A survey of {Z}arankiewicz problem in geometry},  (2024),
  arXiv:2410.03702.

\bibitem{incidenceHigherDim}
J.~Solymosi and T.~Tao, \emph{An incidence theorem in higher dimensions},
  Discrete Comput. Geom. \textbf{48} (2012), no.~2, 255--280. \MR{2946447}

\bibitem{HasseDiagramsChiLarge}
A.~Suk and I.~Tomon, \emph{Hasse diagrams with large chromatic number}, Bull.
  Lond. Math. Soc. \textbf{53} (2021), no.~3, 747--758. \MR{4275086}

\bibitem{GSConjecture2}
D.~P. Sumner, \emph{Subtrees of a graph and the chromatic number}, The theory
  and applications of graphs ({K}alamazoo, {M}ich., 1980), Wiley, New York,
  1981, pp.~557--576. \MR{634555}

\bibitem{ST1983}
E.~Szemer\'{e}di and W.~T. Trotter, \emph{Extremal problems in discrete
  geometry}, Combinatorica \textbf{3} (1983), no.~3-4, 381--392. \MR{729791}

\bibitem{Thomassen83}
C.~Thomassen, \emph{Girth in graphs}, J. Combin. Theory Ser. B \textbf{35}
  (1983), no.~2, 129--141. \MR{733019}

\bibitem{Thurston}
W.~P. Thurston, \emph{The geometry and topology of three-manifolds. {V}ol.
  {IV}}, American Mathematical Society, Providence, RI, [2022] \copyright 2022,
  Edited and with a preface by Steven P. Kerckhoff and a chapter by J. W.
  Milnor. \MR{4554426}

\bibitem{KSTBoxes2021}
I.~Tomon and D.~Zakharov, \emph{Tur\'{a}n-type results for intersection graphs
  of boxes}, Combin. Probab. Comput. \textbf{30} (2021), no.~6, 982--987.
  \MR{4328360}

\bibitem{flipWidth}
S.~Toru\'{n}czyk, \emph{Flip-width: Cops and robber on dense graphs}, 2023 IEEE
  64th Annual Symposium on Foundations of Computer Science (FOCS) (Los
  Alamitos, CA, USA), IEEE Computer Society, nov 2023, pp.~663--700.

\bibitem{mimWidthThesis}
M.~Vatshelle, \emph{New width parameters of graphs}, PhD thesis, University of
  Bergen (2012).

\bibitem{vizing}
V.~G. Vizing, \emph{Critical graphs with given chromatic class}, Diskret.
  Analiz (1965), no.~5, 9--17. \MR{200202}

\bibitem{Lshapesloglog}
B.~Walczak, \emph{Coloring triangle-free {L}-graphs with {$O ( \log \log n )$}
  colors}, European J. Combin. \textbf{117} (2024), Paper no. 103831.
  \MR{4688009}

\bibitem{Walsh2020}
M.~N. Walsh, \emph{The polynomial method over varieties}, Invent. Math.
  \textbf{222} (2020), no.~2, 469--512. \MR{4160873}

\bibitem{W19tw}
D.~Wei{\ss}auer, \emph{In absence of long chordless cycles, large tree-width
  becomes a local phenomenon}, J. Combin. Theory Ser. B \textbf{139} (2019),
  342--352. \MR{4010195}

\bibitem{yolov}
N.~Yolov, \emph{Minor-matching hypertree width}, Proceedings of the
  {T}wenty-{N}inth {A}nnual {ACM}-{SIAM} {S}ymposium on {D}iscrete
  {A}lgorithms, SIAM, Philadelphia, PA, 2018, pp.~219--233. \MR{3775804}

\end{thebibliography}

\end{document}